\documentclass[a4paper,11pt,reqno,oneside]{amsart}

\usepackage{amsmath,amssymb,amsthm,mathtools}
\usepackage{mathrsfs}
\usepackage{enumitem}
\usepackage{xcolor}
\usepackage[scale=0.75,hcentering]{geometry}

\usepackage{microtype}
\usepackage{aliascnt}
\usepackage{hyperref}
\usepackage[capitalize,nameinlink,noabbrev]{cleveref}

\makeatletter
\@ifundefined{subjclassname@2020}{%
\@namedef{subjclassname@2020}{%
\textup{2020} Mathematics Subject Classification}}{}
\makeatother

\definecolor{cobalt}{rgb}{0.0,0.28,0.67}
\hypersetup{
bookmarksopen=true,
pdftitle={Upper Bounds on the Comon Rank Gap of Real Symmetric Third-Order Tensors},
pdfauthor={Jianze Li},
pdfsubject={Real tensor rank and real symmetric tensor rank of cubic tensors},
pdfkeywords={symmetric tensor, CP rank, Waring rank, symmetric rank,
real cubic, tensor rank gap, polarization},
colorlinks=true,
linkcolor=cobalt,
citecolor=cobalt,
urlcolor=cobalt
}

\theoremstyle{plain}

\newtheorem{theorem}{Theorem}[section]

\newaliascnt{corollary}{theorem}
\newtheorem{corollary}[corollary]{Corollary}
\aliascntresetthe{corollary}

\newaliascnt{lemma}{theorem}
\newtheorem{lemma}[lemma]{Lemma}
\aliascntresetthe{lemma}

\newaliascnt{proposition}{theorem}
\newtheorem{proposition}[proposition]{Proposition}
\aliascntresetthe{proposition}

\theoremstyle{definition}

\newaliascnt{definition}{theorem}

\aliascntresetthe{definition}

\newaliascnt{problem}{theorem}
\newtheorem{problem}[problem]{Open Problem}
\aliascntresetthe{problem}

\theoremstyle{remark}

\newaliascnt{remark}{theorem}
\newtheorem{remark}[remark]{Remark}
\aliascntresetthe{remark}

\newaliascnt{example}{theorem}

\aliascntresetthe{example}

\crefname{theorem}{Theorem}{Theorems}
\Crefname{theorem}{Theorem}{Theorems}

\crefname{corollary}{Corollary}{Corollaries}
\Crefname{corollary}{Corollary}{Corollaries}

\crefname{lemma}{Lemma}{Lemmas}
\Crefname{lemma}{Lemma}{Lemmas}

\crefname{proposition}{Proposition}{Propositions}
\Crefname{proposition}{Proposition}{Propositions}

\crefname{definition}{Definition}{Definitions}
\Crefname{definition}{Definition}{Definitions}

\crefname{problem}{Open Problem}{Open Problems}
\Crefname{problem}{Open Problem}{Open Problems}

\crefname{remark}{Remark}{Remarks}
\Crefname{remark}{Remark}{Remarks}

\crefname{example}{Example}{Examples}
\Crefname{example}{Example}{Examples}

\newcommand{\tens}[1]{\boldsymbol{#1}}

\newcommand{\matr}[1]{\mathbf{#1}}

\newcommand{\vect}[1]{\boldsymbol{#1}}

\newcommand{\set}[1]{\mathcal{#1}}

\newcommand{\eqdef}{\stackrel{\sf def}{=}}
\newcommand{\RR}{\mathbb{R}}
\newcommand{\Sym}{\mathbf{Sym}}
\newcommand{\sym}{\operatorname{sym}}
\newcommand{\Span}{\operatorname{span}}
\newcommand{\rank}{\operatorname{rank}}
\newcommand{\Ess}{\operatorname{Ess}}
\newcommand{\Mthree}{M_{3,\RR}}
\newcommand{\Vthree}{V_{3,\RR}}
\newcommand{\gthree}{g_3}
\newcommand{\DeltaT}{\Delta}
\newcommand{\Ce}{C_{ec}}
\newcommand{\Cd}{C_{dc}}
\newcommand{\Cdsharp}{C_{dc}^{\sharp}}

\allowdisplaybreaks

\title[Upper Bounds on the Comon Rank Gap]
{Upper Bounds on the Comon Rank Gap of Real Symmetric Third-Order Tensors}

\author{Jianze Li}
\address{School of Science, Sun Yat-sen University, Shenzhen, Guangdong, China}
\email{lijianze@mail.sysu.edu.cn}

\date{\today} 

\begin{document}

\begin{abstract}
Let $\tens{T}$ be a nonzero real symmetric third-order tensor.  Denote by
$r(\tens{T})$ its CP rank, by $s(\tens{T})$ its symmetric rank, by
$d(\tens{T})$ its essential dimension, and put
$e(\tens{T})=r(\tens{T})-d(\tens{T})$.  We establish two complementary
upper bounds for the Comon rank gap
\[
\DeltaT(\tens{T})\eqdef s(\tens{T})-r(\tens{T})
\leq\min\bigl\{
\Ce\bigl(e(\tens{T})\bigr),\ \Cd\bigl(d(\tens{T})\bigr)\bigr\}.
\]
The excess-controlled function $\Ce(e)$ is zero for $0\leq e\leq3$
and equals $9e-3$ for $e\geq4$. 
In particular, the equality
$\Ce\bigl(e(\tens{T})\bigr)=0$ for $0\leq e\leq3$ extends Friedland's
flattening-rank criterion for real symmetric third-order tensors by two
additional rank layers. 
The dimension-controlled function $\Cd(d)$ is zero for $d=0$,
equals $(d^2-4d+10)/3$ for $1\leq d\leq18$, and equals
$d^2/4+d/4+3/2$ for $d\geq19$. 
The excess estimate follows from the alternating equations and
decomposable-bivector support geometry; a low-wedge descent gives the
vanishing range $e\leq3$.  The dimension estimate combines a transverse
polar-completion bound with a shifted estimate for the maximum real
Waring rank.  Thus the two bounds measure different features of the
same tensor and remain effective in complementary regimes.
\end{abstract}

\maketitle

\section{Introduction}
\label{sec:introduction}

% The \emph{CP rank} of a tensor is the least length of a decomposition into simple
% tensors, as introduced in  \cite{hitchcock1927expression,kruskal1977three}.  
The \emph{CP rank} of a tensor is the least length of a decomposition
into simple tensors, going back to Hitchcock
\cite{hitchcock1927expression}; see also Kruskal's foundational work
on rank and uniqueness \cite{kruskal1977three}.
For a
symmetric tensor one may impose the additional requirement that every
simple summand be symmetric.  The resulting symmetric rank is the \emph{Waring
rank} of the corresponding homogeneous polynomial.  These notions are
fundamental in multilinear algebra and algebraic geometry; see, for
example, \cite{comon2008symmetric,kolda2009tensor,landsberg2012tensors,comon2014tensors}.
Beyond their intrinsic algebraic interest, CP rank and symmetric rank also play important roles in \emph{signal processing, statistics}, and \emph{machine
learning}.  
For example, the uniqueness of tensor rank decompositions
provides criteria for identifiability in \emph{independent component analysis} (ICA) \cite{kruskal1977three,comon1994independent,Como10:book},
while symmetric tensor decompositions are used for source estimation in the ICA model \cite{comon1994independent,cardoso1996jacobi,Lathauwer01:ICA,LUC2017globally,ULC2019}.

Let $\set{U}$ be a real vector space, and denote by $\Sym^3(\set{U})$ the space of third-order symmetric tensors \cite{comon2008symmetric}. For $\tens{T}\in\Sym^3(\set{U})$, let $r(\tens{T})\eqdef\rank(\tens{T})$ and
$s(\tens{T})\eqdef\rank_{\mathrm S}(\tens{T})$ denote its CP rank and symmetric rank, respectively. We define the \emph{Comon rank gap} of $\tens{T}$ by
$$\DeltaT(\tens{T})\eqdef s(\tens{T})-r(\tens{T}).$$ 
Because every symmetric rank-one tensor is rank one, one always has
$\DeltaT(\tens{T})\geq0$.  Let $\Ess(\tens{T})$ be the smallest subspace $\set{W}\subseteq \set{U}$
such that $\tens{T}\in\Sym^3(\set{W})$, and set
$d(\tens{T})\eqdef\dim\Ess(\tens{T}).$ 
The first flattening has rank $d(\tens{T})$, so $r(\tens{T})\geq d(\tens{T})$.  
We define the \emph{rank excess} of $\tens{T}$ by
$e(\tens{T})\eqdef r(\tens{T})-d(\tens{T})$, 
which measures how far its CP rank exceeds the flattening lower bound.

A central problem in this area is \emph{Comon's conjecture}
\cite{comon2004canonical,comon2008symmetric,comon2014tensors}, which
asks whether
$r(\tens{T})=s(\tens{T})$ 
holds for every real or complex symmetric tensor $\tens{T}$. A proposed complex
counterexample \cite{shitov2018counterexample} was subsequently shown to have an
error \cite{draisma2024erratum}; new counterexamples over the complex numbers
were recently reported in \cite{shitovfurther,Lovitz26:arxiv}.
% A more recent preprint claims counterexamples over every
% infinite field of characteristic different from two and three, including the
% order-three case \cite{shitovhigher}. We distinguish this preprint claim from
% the peer-reviewed status following \cite{draisma2024erratum}. 
Over the reals, counterexamples are known in orders four and six
\cite{shitov2020comon,wang2023lower}, while order-three counterexamples are 
reported in \cite{shitovhigher,Lovitz26:arxiv}.
% Shitov constructed an order-four real
% example, and Wang and Seigal developed flattening-based lower bounds and an
% order-six example \cite{shitov2020comon,wang2023lower}; the more recent
% field-uniform claim is given in \cite{shitovhigher}. These works establish
% that the gap need not vanish, but they do not supply an upper bound controlled
% by the essential dimension or by the rank excess. 
These works mainly concern constructions and lower bounds, rather than upper
bounds in terms of the essential dimension or the rank excess. The conjecture
has nevertheless been established under several additional hypotheses. Over
$\mathbb{C}$, the foundational work \cite{comon2008symmetric} proves equality in
basic low-symmetric-rank cases and under suitable decomposition conditions.
Zhang, Huang, and Qi prove equality when the CP rank does not exceed the tensor
order \cite{zhang2016comon}, while Friedland's flattening criterion covers
$r(\tens{T})\leq d(\tens{T})+1$ \cite{friedland2016remarks}. Other positive
results over $\mathbb{C}$ include tensors of sufficiently low symmetric rank
\cite{ballico2011lowrank}, tensors on tangential varieties of Veronese varieties
\cite{ballico2013tensor}, and the cubic-surface case; Seigal moreover proves
equality for complex cubic tensors of symmetric rank at most seven
\cite{seigal2020ranks}.
% For structured tensors, Nie and Ye obtain
% the equality of CP, symmetric, and Vandermonde ranks for generic Hankel tensors
% of order three (and of even order) \cite{nie2019hankel}. 
For structured tensors, Nie and Ye obtain, over $\mathbb{C}$, the equality of
CP, symmetric, and Vandermonde ranks for generic Hankel tensors of order three
(and of even order) \cite{nie2019hankel}. 
Over $\mathbb{C}$, Ballico et al.\ establish cactus-rank
lower bounds and derive sufficient conditions for equality
of tensor rank and symmetric rank 
\cite{ballico2018bounds}. Related rank-equality results include the equality of
ordinary and symmetric nuclear norms for every symmetric tensor over both
$\RR$ and $\mathbb{C}$ \cite{friedland2018nuclear}, and the equality of
$G$-stable and symmetric $G$-stable ranks over perfect fields
\cite{jiang2024gstable}. The border-rank analogue, although governed by related
secant-variety questions, is a distinct problem; known positive cases and
recent advances appear in
\cite{landsberg2010ranks,buczynski2013determinantal,mandziuk2024symmetrization}.
Other variants of Comon's conjecture include the partially symmetric setting
\cite{gesmundo2019partially}, formulations over arbitrary fields
\cite{zheng2020arbitrary,huang2022rank}, the orthogonal version \cite{LUC2018},
and the best rank one approximation problem
\cite{friedland2013best,zhang2012best}. We refer to
\cite{casarotti2020comon} for a broader survey.
% We refer to \cite{casarotti2020comon} for a broader survey of the
% problem and its developments.

% The real Waring problem has features absent over an algebraically closed
% field, including several typical ranks and rank boundaries.  Useful
% background and explicit cubic calculations may be found in
% \cite{banchi2015ternary,carlini2016reducible,
% bernardi2018typical,michalek2017ternary,michalek2018boundaries}; general
% upper and lower bounds are developed in
% \cite{landsberg2010ranks,jelisiejew2014upper,
% blekherman2015maximum,blekherman2016realrank,wang2023lower}.

Note that Comon's conjecture is equivalent to the assertion that
the Comon rank gap $\DeltaT(\tens{T})=0$ for every symmetric tensor $\tens{T}$, 
and real counterexamples
\cite{shitov2020comon,wang2023lower,shitovhigher}
show that $\DeltaT(\tens{T})$ can be nonzero over the reals. 
% For the real third-order tensors considered here, Comon's conjecture
% would assert that the Comon rank gap
% $\DeltaT(\tens{T})=0$ for every
% $\tens{T}\in\Sym^3(\set{U})$. 
This naturally leads to the following quantitative question for real symmetric tensors:
\begin{equation}
\textit{How large can $\DeltaT(\tens{T})$ be in terms of
$d(\tens{T})$ and $e(\tens{T})$?}
\label{eq-question-comon}
\end{equation}
This question quantifies the extent to which Comon's conjecture can
fail and seeks to identify the structural parameters governing such
failure. In particular, a bound in terms of $d(\tens{T})$ and $e(\tens{T})$ determines
whether the Comon rank gap remains uniformly controlled within natural
structural classes of symmetric tensors. 
It is also closely related to the computation and estimation of the
CP rank and symmetric rank of symmetric tensors
\cite{kolda2009tensor,brachat2010symmetric,
bernardi2011computing,hillar2013most}. Maximum Waring-rank estimates provide
general, if non-adaptive, bounds for the gap. The Alexander--Hirschowitz theorem
determines the generic complex Waring rank
\cite{hirschowitz1995polynomial,brambilla2008alexander}, and the maximum-rank
theorem of Blekherman and Teitler bounds the maximum real Waring rank by twice
the generic complex rank \cite{blekherman2015maximum}. Exact or sharper
information in small cubic formats is available for binary and ternary forms
\cite{reznick2013length,banchi2015ternary}. Together with
$r(\tens{T})\leq s(\tens{T})\leq\Mthree(d(\tens{T}))$,
where $\Mthree(d(\tens{T}))$ denotes the maximum real Waring rank
in $d(\tens{T})$ variables, as defined in \eqref{eq:m_3_R},
these results give the standard bound
$\DeltaT(\tens{T})\leq\Mthree(d(\tens{T}))-r(\tens{T})$; 
direct polarization instead gives $\DeltaT(\tens{T})\leq3r(\tens{T})$. %Neither estimate reflects the finer parameter $e(\tens{T})$, which is the point of departure for the results below.
These estimates also serve as the starting point for our analysis in this paper.

In this paper, we address the quantitative question
\eqref{eq-question-comon} for real symmetric third-order tensors.
All tensors to which our results apply, and all decompositions defining
their ranks, are over $\RR$.  
Our two
main bounds answer \eqref{eq-question-comon} from complementary
directions. 
The first bound, \cref{thm:rank-excess-main}, is controlled solely by
the rank excess $e(\tens{T})$.  It proves
$\DeltaT(\tens{T})=0$ for $e(\tens{T})\leq3$ and
$\DeltaT(\tens{T})\leq9e(\tens{T})-3$ for $e(\tens{T})\geq4$.
After displayed cubes are deleted, the alternating equations of
symmetry and a wedge-support count give the linear estimate.  For
$e\leq3$, low-wedge partialization and saturated descent give
$s(\tens{T})=r(\tens{T})$. 
The second bound, \cref{thm:essential-dimension-main}, is controlled
solely by the essential dimension $d(\tens{T})$.  Its first ingredient
is the rank-sensitive polar estimate
\[
\DeltaT(\tens{T})
\leq\Mthree(d(\tens{T})-1)-d(\tens{T})+1.
\]
A quadratic polar is completed by cubes transverse to its defining
hyperplane, while a minimum CP decomposition supplies a hyperplane
containing sufficiently many factor columns.  Its second ingredient is
the shifted maximum-rank estimate
$\Mthree(d)\leq2\gthree(d-1)+1$, obtained from a transverse polar
degeneration and a regular local completion.  Together they give the
piecewise quadratic function $\Cd$.  Directional avoidance and related
derivative-reduction arguments occur in
\cite{BialynickiBirulaSchinzel2008,BialynickiBirulaSchinzel2011corrigendum,jelisiejew2014upper}; the present
form is adapted to the Comon gap.  The generic-rank input is the
Alexander--Hirschowitz theorem
\cite{hirschowitz1995polynomial,brambilla2008alexander}.

The paper is organized accordingly.  \Cref{sec:preliminaries} fixes
notation and recalls the rank estimates used later, and
\cref{sec:main-results} states the two main theorems together with
their supporting propositions.  \Cref{sec:proof-excess-general}
proves the general linear excess bound,
\cref{sec:proof-small-excess} proves the zero-gap result for excess at
most three, and \cref{sec:proof-dimension} proves the two
dimension-controlled propositions.  \Cref{sec:discussion} summarizes
the consequences and records the remaining absolute-gap problem.

\section{Preliminaries}
\label{sec:preliminaries}

\subsection{Ranks and essential variables}
All vector spaces are finite-dimensional.  Tensors are denoted by bold letters such as $\tens{T}$ and $\tens{Q}$,
matrices and linear maps by $\matr{A}$ and $\matr{M}$, vectors and
bivectors by $\vect{a}$ and $\vect{p}$, and sets and vector spaces by
calligraphic letters.
We use $\tens{T}$ for the tensor under consideration and $\tens{R}$
for a residual; auxiliary tensors are named locally.  The functions
$\Ce$, $\Cd$, and $\Cdsharp$ denote gap bounds, whereas $\Mthree$
denotes a maximum rank and $\Vthree$ its upper bound.
Let $\set{U}_1$, $\set{U}_2$ and $\set{U}_3$ be real vector spaces.
For $\tens{T}\in
\set{U}_1\otimes\set{U}_2\otimes\set{U}_3$, its \emph{CP rank} \cite{hitchcock1927expression,kruskal1977three,kolda2009tensor,comon2014tensors,comon2020topology} is defined as 
\[
r(\tens{T})=\rank(\tens{T})\eqdef\min\left\{q:\tens{T}=\sum_{i=1}^q
\lambda_i\vect{a}_i\otimes\vect{b}_i\otimes\vect{c}_i,\
\vect{a}_i\in\set{U}_1,
\vect{b}_i\in\set{U}_2,
\vect{c}_i\in\set{U}_3, \lambda_i\in\RR\right\}.
\]
Let $\set{U}$ be a real vector space, and denote by 
\[
\Sym^3(\set{U})
\eqdef
\left\{
\tens{T}\in\set{U}^{\otimes 3}:
\sigma\cdot\tens{T}=\tens{T}
\text{ for every }\sigma\in\mathfrak{S}_3
\right\}
\] 
the space of third-order symmetric tensors over $\set{U}$,
where $\mathfrak{S}_3$ acts by permuting the three tensor factors.
For $\tens{T}\in\Sym^3(\set{U})$, its \emph{symmetric rank} \cite{kolda2009tensor,comon2008symmetric} is defined as 
\[
s(\tens{T})=\rank_{\mathrm S}(\tens{T})
\eqdef\min\left\{q:\tens{T}=\sum_{i=1}^q
\lambda_i\vect{u}_i^{\otimes3},\ 
\vect{u}_i\in\set{U}, \lambda_i\in\RR\right\}.
\] 
For a symmetric tensor $\tens{T}\in\Sym^3(\set{U})$, its
\emph{essential space}, denoted by $\Ess(\tens{T})$, is the smallest
linear subspace $\set{W}\subseteq\set{U}$ such that
$\tens{T}\in\Sym^3(\set{W})$. 
We call $\tens{T}$ \emph{concise} in $\set{U}$ if
$\Ess(\tens{T})=\set{U}$. 
For a quadratic tensor $\tens{Q}\in\Sym^2(\set{U})$,
$\Ess(\tens{Q})$ is defined analogously; its dimension equals
the matrix rank of $\tens{Q}$.

Write
\(
  \matr{F}_{\tens{T}}:\set{U}^*\to\set{U}^{\otimes2},\
  \alpha\mapsto\alpha\mathbin{\lrcorner}\tens{T}.
\)
This contraction map will clarify the support and rank arguments below.
We record the following elementary result. 

\begin{lemma}[Contraction and essential support]
\label{lem:contraction-essential-support}
For $\tens{T}\in\Sym^3(\set{U})$, we have
\[
 \Sym^3(\set{U})=\Span\{\vect{u}^{\otimes3}:\vect{u}\in\set{U}\},
 \ \ 
 \Ess(\tens{T})=(\ker\matr{F}_{\tens{T}})^\perp,
 \quad
 d(\tens{T})=\rank\matr{F}_{\tens{T}}\le r(\tens{T})\le s(\tens{T}).
\]
\end{lemma}

% \begin{proof}
% % LEAN: mem_span_iff_permutationSymmetric; essentialSpace_isLeast/finrank;
% % actual_flatteningDim_le_cpRank; actual_cpRank_le_symmetricRank.
% The polarization identity
% \[
% \begin{aligned}
% 6\sym(\vect{u}\otimes\vect{v}\otimes\vect{w})
% ={}(\vect{u}+\vect{v}+\vect{w})^3
%  -(\vect{u}+\vect{v})^3-(\vect{u}+\vect{w})^3-(\vect{v}+\vect{w})^3
%  +\vect{u}^3+\vect{v}^3+\vect{w}^3
% \end{aligned}
% \]
% gives the spanning assertion.  For $\set{W}\subseteq\set{U}$,
% $\tens{T}\in\Sym^3(\set{W})$ implies
% $\set{W}^\perp\subseteq\ker\matr{F}_{\tens{T}}$.
% Conversely, suppose this inclusion holds and let $\matr{\Pi}:\set{U}\to\set{W}$
% be a projection, viewed as an endomorphism of $\set{U}$.  For every $\alpha$,
% $\alpha\circ\matr{\Pi}-\alpha\in\set{W}^\perp$, so we obtain
% \[
%  \matr{F}_{(\matr{\Pi}\otimes\operatorname{id}^{\otimes2})\tens{T}}(\alpha)
%  =\matr{F}_{\tens{T}}(\alpha\circ\matr{\Pi})
%  =\matr{F}_{\tens{T}}(\alpha).
% \]
% The tensor-to-flattening map is injective, as is seen by evaluating on a
% basis of $\set{U}^*$.  Thus projection in the first mode fixes $\tens{T}$;
% symmetry gives the same assertion in the other modes.  Hence
% $\matr{\Pi}^{\otimes3}\tens{T}=\tens{T}\in\Sym^3(\set{W})$.
% Taking annihilators proves the formula for $\Ess(\tens{T})$.
% Rank--nullity gives $d(\tens{T})=\rank\matr{F}_{\tens{T}}$.
% A CP decomposition of length $r$ gives a sum of $r$ linear maps of rank
% at most one, and a cube is a CP term, proving the remaining inequalities.
% \end{proof}

We recall the following standard ambient-space invariance property of
CP rank and symmetric rank; see, for example,
\cite{carlini2006reducing} for the corresponding reduction to essential
variables. Since the result follows directly by applying a linear
projection from the ambient space onto the smaller subspace, we omit its proof.

\begin{lemma}[Ambient-space invariance]
\label{lem:concise-reduction}
Let $\set{W}\subseteq\set{U}$ be real vector spaces. If
$\tens{T}\in\Sym^3(\set{W})\subseteq\Sym^3(\set{U})$, then its CP rank and symmetric rank
are the same whether computed in $\set{W}$ or in $\set{U}$.
In particular, one may
replace $\set{U}$ by $\Ess(\tens{T})$ without changing either rank.
\end{lemma}

We next recall the following standard properties of a minimum CP decomposition; see, e.g., \cite{landsberg2012tensors,comon2008symmetric}. They follow directly
from the usual substitution and contraction arguments, so we also omit the details.

\begin{lemma}[Minimum CP decompositions]
\label{lem:minimum-cpd}
Let $\set{U}_1,\set{U}_2,\set{U}_3$ be real vector spaces, let
$\tens{T}\in\set{U}_1\otimes\set{U}_2\otimes\set{U}_3$, and suppose that
\[
\tens{T}=\sum_{i=1}^{r(\tens{T})}
\vect{a}_i\otimes\vect{b}_i\otimes\vect{c}_i
\]
is a minimum CP decomposition, where
$\vect{a}_i\in\set{U}_1$, $\vect{b}_i\in\set{U}_2$, and
$\vect{c}_i\in\set{U}_3$. Then each of the three families
\[
(\vect{a}_i\otimes\vect{b}_i)_{i=1}^{r(\tens{T})},\ \ 
(\vect{b}_i\otimes\vect{c}_i)_{i=1}^{r(\tens{T})},\ \ 
(\vect{c}_i\otimes\vect{a}_i)_{i=1}^{r(\tens{T})},
\]
is linearly independent. If $\set{U}_1=\set{U}_2=\set{U}_3=\set{U}$ and
$\tens{T}\in\Sym^3(\set{U})$, then
the spans of the three vector families
$(\vect{a}_i)_{i=1}^{r(\tens{T})}$,
$(\vect{b}_i)_{i=1}^{r(\tens{T})}$ and
$(\vect{c}_i)_{i=1}^{r(\tens{T})}$
all coincide with $\Ess(\tens{T})$. Moreover, the rank-one tensors
$(\vect{a}_i\otimes\vect{b}_i\otimes\vect{c}_i)_
{i=1}^{r(\tens{T})}$
are linearly independent.
\end{lemma}

\subsection{Deleting displayed cubes}
\label{subsec:displayed-cube-deletion}

A summand $\vect{a}_i\otimes \vect{b}_i\otimes \vect{c}_i$ in a CP decomposition of a
symmetric tensor will be called a \emph{displayed cube} if
$\vect{a}_i,\vect{b}_i,\vect{c}_i$ are proportional.
The following elementary reduction removes all such summands; its
proof uses minimality and compares the ranks of the residual factor
matrices.

\begin{lemma}[Deletion of displayed cubes]
\label{lem:displayed-cube-deletion}
Let $\tens{T}\in\Sym^3(\set{U})$, and fix a minimum CP decomposition
\[
\tens{T}=\sum_{i=1}^{r(\tens{T})}\vect{a}_i\otimes \vect{b}_i\otimes \vect{c}_i.
\]
Let
$\set{C}
=\{i:\dim\Span\{\vect{a}_i,\vect{b}_i,\vect{c}_i\}=1\},$ 
and write
$$\tens{M}=\sum_{i\in\set{C}}\vect{a}_i\otimes \vect{b}_i\otimes \vect{c}_i, \ \ 
\tens{R}=\sum_{i\notin\set{C}}\vect{a}_i\otimes \vect{b}_i\otimes \vect{c}_i.$$ 
Then $\tens{M}$ is a sum of $|\set{C}|$ real cubes, $\tens{R}$ is symmetric, and
the displayed decomposition of $\tens{R}$ is minimum and cube-free.  Moreover, we have 
$r(\tens{R})=r(\tens{T})-|\set{C}|$ and  $0\le d(\tens{T})-d(\tens{R})\le|\set{C}|,$ 
and consequently
\[
e(\tens{R})\le e(\tens{T}),
\ \ 
\DeltaT(\tens{T})\le\DeltaT(\tens{R}).
\]
Here $r(\tens{0})=s(\tens{0})=d(\tens{0})=0$.
\end{lemma}

\begin{proof}
Every summand indexed by $\set{C}$ is a scalar multiple of a real
cube.  Hence $\tens{M}$ and $\tens{R}=\tens{T}-\tens{M}$ are symmetric, and the displayed
decomposition of $\tens{R}$ is cube-free.  If $\tens{R}$ admitted a CP decomposition
with fewer than $r(\tens{T})-|\set{C}|$ terms, adjoining the displayed
cubes in $\tens{M}$ would give a CP decomposition of $\tens{T}$ shorter than the
chosen one.  Thus $r(\tens{R})=r(\tens{T})-|\set{C}|$.  By \cref{lem:minimum-cpd}, each factor matrix of the original decomposition has rank
$d(\tens{T})$, whereas the corresponding residual factor matrix has rank
$d(\tens{R})$.  Deleting $|\set{C}|$ columns lowers matrix rank by at most
$|\set{C}|$, and hence $0\le d(\tens{T})-d(\tens{R})\le|\set{C}|$.
Therefore, we have
\[
e(\tens{R})
=r(\tens{T})-|\set{C}|-d(\tens{R})
=e(\tens{T})-|\set{C}|+d(\tens{T})-d(\tens{R}) 
\le e(\tens{T}).
\]
Finally, we have 
$s(\tens{T})\le |\set{C}|+s(\tens{R})$ and $r(\tens{T})=|\set{C}|+r(\tens{R})$, 
so $\DeltaT(\tens{T})\le\DeltaT(\tens{R})$.
\end{proof}

A transverse rank-one polar gives an exact deletion even when its cube
is not displayed in the chosen minimum CP decomposition.

\begin{lemma}[Transverse rank-one deletion]
\label{lem:transverse-rank-one-deletion}
Let $\tens{T}\in\Sym^3(\set{V})$ be concise.  Suppose
$\alpha\mathbin{\lrcorner}\tens{T}=q\vect{\ell}^{2}$ with
$q\ne0$ and $\alpha(\vect{\ell})\ne0$, and put
$\tens{R}=\tens{T}-q\vect{\ell}^3/\alpha(\vect{\ell})$.
Then
\[
 \tens{R}\in\Sym^3(\ker\alpha),\ \ 
 r(\tens{R})=r(\tens{T})-1,\ \ 
 d(\tens{R})=d(\tens{T})-1.
\]
In particular, $e(\tens{R})=e(\tens{T})$ and
$\DeltaT(\tens{T})\le\DeltaT(\tens{R})$.
\end{lemma}

\begin{proof}
% LEAN: transverse_rank_one_deletion (including selection/reordering
% of the CP term); transverse_deletion_invariants verifies e and Delta.
% essentialSpace_finrank supplies the actual dimension.
Since $\alpha\mathbin{\lrcorner}\tens{R}=\tens{0}$, symmetry and
\cref{lem:contraction-essential-support} give
$\tens{R}\in\Sym^3(\ker\alpha)$.
Choose a minimum CP decomposition of $\tens{T}$.
The nonzero polar forces some first factor $\vect{a}_k$ to satisfy
$\alpha(\vect{a}_k)\ne0$.  Define
\[
 \matr{\Pi}(\vect{v})=\vect{v}-\frac{\alpha(\vect{v})}{\alpha(\vect{\ell})}\vect{\ell},
 \ \ 
 \matr{P}(\vect{v})=\vect{v}-\frac{\alpha(\vect{v})}{\alpha(\vect{a}_k)}\vect{a}_k.
\]
Both maps restrict to the identity on $\ker\alpha$, while
$\matr{\Pi}(\vect{\ell})=\vect{0}$ and $\matr{P}(\vect{a}_k)=\vect{0}$.  Consequently,
\[
 (\matr{P}\otimes\matr{\Pi}\otimes\matr{\Pi})\tens{T}=\tens{R},
 \ \  r(\tens{R})\le r(\tens{T})-1\le r(\tens{R}).
\]
The second inequality follows by adjoining the removed cube.
Also $d(\tens{R})\le d(\tens{T})-1$, whereas
$\Ess(\tens{T})\subseteq\RR\vect{\ell}+\Ess(\tens{R})$
gives the reverse inequality.  The last assertions follow from these
identities and $s(\tens{T})\le1+s(\tens{R})$.
\end{proof}

\subsection{Symmetrization identities}

We identify $\Sym^k(\set{U})$ with the space of homogeneous
polynomial functions of degree $k$ on $\set{U}^*$.
Thus, vectors in $\set{U}$ are viewed as linear forms on
$\set{U}^*$.
In expressions for tensor subspaces, a vector in a tensor factor
denotes its linear span; for example,
$\set{H}\otimes\vect{x}\otimes\vect{x}$ abbreviates
$\set{H}\otimes\RR\vect{x}\otimes\RR\vect{x}$.
%We distinguish zero objects typographically: $0$ denotes a scalar,$\vect{0}$ a vector, multivector, matrix, or a tensor. 
Dual functionals are denoted by unbold Greek letters.  For
$\alpha\in\set{U}^*$, the contraction
$\alpha\mathbin{\lrcorner}\tens{T}$ acts on the first tensor
factor.  On cubics it corresponds to $\frac13\partial_\alpha$;
in particular,
$\alpha\mathbin{\lrcorner}\vect{v}^3
=\alpha(\vect{v})\vect{v}^2$.

For $\vect{a},\vect{b},\vect{c},\vect{u},\vect{v}\in\set{U}$, we
use the normalized symmetrization
\[
\sym(\vect{a}\otimes\vect{b}\otimes\vect{c})
\eqdef
\frac{1}{6}\sum_{\pi\in\mathfrak S_3}
\vect{f}_{\pi(1)}\otimes\vect{f}_{\pi(2)}
\otimes\vect{f}_{\pi(3)},
\]
where
$(\vect{f}_1,\vect{f}_2,\vect{f}_3)
=(\vect{a},\vect{b},\vect{c})$.
Under the identification of symmetric tensors with homogeneous
polynomials, this tensor corresponds to the cubic product
$\vect{a}\vect{b}\vect{c}$. The following real polarization identities
will be used repeatedly:
\begin{align}
24\vect{a}\vect{b}\vect{c}
={}&(\vect{a}+\vect{b}+\vect{c})^3
+(\vect{a}-\vect{b}-\vect{c})^3
+(-\vect{a}+\vect{b}-\vect{c})^3
+(-\vect{a}-\vect{b}+\vect{c})^3,
\label{eq:four-cube}\\
6\vect{u}\vect{v}^{2}
={}&(\vect{u}+\vect{v})^3
+(\vect{u}-\vect{v})^3
-2\vect{u}^{3}.
\label{eq:three-cube}
\end{align}
It follows from \eqref{eq:four-cube} that the symmetrization of an
arbitrary simple tensor has symmetric rank at most four. If two of its
factors are proportional, \eqref{eq:three-cube} improves this bound to
three.

\subsection{Generic and maximum Waring ranks}

Set $\gthree(0)=0$.  For $d\geq1$, let $\gthree(d)$ denote
the \emph{generic complex Waring rank} of a cubic form in $d$ variables; see,
for example, \cite{hirschowitz1995polynomial,brambilla2008alexander}.
The Alexander--Hirschowitz theorem gives
\begin{equation}
\gthree(d)=
\begin{cases}
\left\lceil\dfrac{(d+1)(d+2)}{6}\right\rceil,
& d\neq5,\\[6pt]
8,& d=5.
\end{cases}
\label{eq:generic-cubic-rank}
\end{equation}
Thus, the case of cubic forms in five variables is the unique
exception to the expected generic-rank formula in degree three. 
For $d\geq1$, define the maximum real Waring rank
\begin{equation}\label{eq:m_3_R}
\Mthree(d)\eqdef
\max\{s(\tens{T}):\tens{T}\in\Sym^3(\RR^d)\},
\end{equation}
and set
$\Mthree(0)\eqdef0$ \cite{blekherman2015maximum}.
This maximum is finite because cubes of real linear forms span the
finite-dimensional space $\Sym^3(\RR^d)$.  Blekherman and Teitler prove
$\Mthree(d)\leq2\gthree(d)$ \cite[Theorems~2 and~3]{blekherman2015maximum};
in particular, the exceptional value $\gthree(5)=8$ gives
$\Mthree(5)\leq16$.  The stronger internal estimate
$\Mthree(5)\leq11$ follows below from
\cref{prop:shifted-maximum-rank}. 
For later comparison, set $\Vthree(0)=0$ and
$\Vthree(d)\eqdef2\gthree(d-1)+1$ for $d\geq1$.  The shifted bound
$\Mthree(d)\leq\Vthree(d)$ will be proved in
\cref{prop:shifted-maximum-rank}.

\section{Main results}
\label{sec:main-results}

The two principal results are arranged by the parameter controlling the
Comon rank gap.  The first uses the rank excess and the second uses the
essential dimension.  We state immediately after each theorem the
propositions from which it follows.

\subsection{Bounds controlled by rank excess}

For $e\in\mathbb Z_{\geq0}$ define
\begin{equation}
\label{eq:def-C-excess}
\Ce(e)=
\begin{cases}
0,&0\leq e\leq3,\\
9e-3,&e\geq4.
\end{cases}
\end{equation}

The first main theorem bounds the gap by the rank excess alone.
\begin{theorem}[Rank-excess bound]
\label{thm:rank-excess-main}
For every real vector space $\set{U}$ and every
$\tens{T}\in\Sym^3(\set{U})$,
\[
\DeltaT(\tens{T})\leq \Ce\bigl(e(\tens{T})\bigr).
\]
\end{theorem}

For arbitrary $e$, the basic wedge-support argument gives the following
uniform estimate.

\begin{proposition}[General rank-excess estimate]
\label{prop:linear-excess-main}
For every real vector space $\set{U}$ and every
$\tens{T}\in\Sym^3(\set{U})$,
\[
\DeltaT(\tens{T})
\leq\max\{0,9e(\tens{T})-3\}.
\]
\end{proposition}

The zero-gap range is supplied by the following proposition.

\begin{proposition}[Zero gap through excess three]
\label{prop:excess-three-zero}
For every real vector space $\set{U}$ and every
$\tens{T}\in\Sym^3(\set{U})$,
\[
0\leq e(\tens{T})\leq3
\quad\Longrightarrow\quad
\DeltaT(\tens{T})=0.
\]
\end{proposition}

\subsection{Bounds controlled by essential dimension}

Define $\Cd:\mathbb Z_{\geq0}\to\RR_{\geq0}$ by
\begin{equation}
\label{eq:def-C-dimension}
\Cd(d)=
\begin{cases}
0,&d=0,\\
\dfrac{d^2-4d+10}{3},&1\leq d\leq18,\\[5pt]
\dfrac{d^2}{4}+\dfrac d4+\dfrac32,&d\geq19.
\end{cases}
\end{equation}
The second main theorem bounds the gap by the essential dimension alone.
\begin{theorem}[Essential-dimension bound]
\label{thm:essential-dimension-main}
For every real vector space $\set{U}$ and every
$\tens{T}\in\Sym^3(\set{U})$,
\[
\DeltaT(\tens{T})\leq \Cd\bigl(d(\tens{T})\bigr).
\]
\end{theorem}

The two branches in \eqref{eq:def-C-dimension} are supplied by the distinct estimates below.
The first dimension estimate combines a transverse polar with a
hyperplane containing $d-1$ factor columns.

\begin{proposition}[Maximum-Waring-rank bound]
\label{prop:maximum-waring-gap}
Let $\set{U}$ be a real vector space and let
$\tens{T}\in\Sym^3(\set{U})$ be nonzero.  Put
$d=d(\tens{T})$.  Then
\begin{equation}
\label{eq:uniform-polar-bound}
\DeltaT(\tens{T})
\leq\Mthree(d-1)-d+1
\leq\frac{d^2-4d+10}{3}.
\end{equation}
\end{proposition}

For the second estimate, put $\Cdsharp(0)=0$ and, for $m\ge1$, set
\begin{equation*}
\Cdsharp(m)=
\min\left\{
\Vthree(m-1)-m+1,\,
\left\lfloor\frac{3\Vthree(m)}4\right\rfloor
\right\}.
\end{equation*}

The second proposition compares the maximum-rank and polarization branches.
\begin{proposition}[Quadratic dimension bound]
\label{prop:dimension-quadratic}
Let $\set{U}$ be a real vector space of dimension $n\geq1$, and let
$\tens{T}\in\Sym^3(\set{U})$.  Put $d=d(\tens{T})$.  Then
\[
\DeltaT(\tens{T})\leq \Cdsharp(d)
\leq\frac{d^2}{4}+\frac d4+\frac32
\leq\frac{n^2}{4}+\frac n4+\frac32.
\]
\end{proposition}

\subsection{Combined form}

Taking the minimum of the two main bounds gives the following estimate;
the standard bounds are
$\DeltaT(\tens{T})\le3r(\tens{T})$ and
$\DeltaT(\tens{T})\le\Mthree(d(\tens{T}))-r(\tens{T})$.

\begin{corollary}[Joint excess--dimension estimate]
\label{cor:unified-bound}
Let $\set{U}$ be a real vector space and let
$\tens{T}\in\Sym^3(\set{U})$ be nonzero.  Then
\[
\DeltaT(\tens{T})
\leq\min\left\{
\Ce\bigl(e(\tens{T})\bigr),
\Cd\bigl(d(\tens{T})\bigr)
\right\}.
\]
\end{corollary}

\begin{proof}
Take the minimum of
\cref{thm:rank-excess-main,thm:essential-dimension-main}.
\end{proof}

\section{The general rank-excess bound}
\label{sec:proof-excess-general}

This section proves \cref{prop:linear-excess-main}.  We first record a
cube-free residual estimate and then derive the claimed uniform bound.

\subsection{The general linear bound}

\Cref{lem:displayed-cube-deletion} shows that deleting all
displayed cubes from a minimum CP decomposition produces a symmetric,
cube-free residual $\tens{R}$ with a minimum displayed decomposition, while
\[
e(\tens{R})\le e(\tens{T}),
\ \ 
\DeltaT(\tens{T})\le\DeltaT(\tens{R}).
\]
It therefore remains to prove the following residual statement.

\begin{proposition}[Cube-free residual bound]
\label{prop:cube-free-residual}
Let $\tens{R}$ be a nonzero real symmetric cubic tensor, set
$\set{V}=\Ess(\tens{R})$, and suppose that
\begin{equation}
\label{eq:cube-free-cpd}
\tens{R}=\sum_{i=1}^{r(\tens{R})}\vect{a}_i\otimes \vect{b}_i\otimes \vect{c}_i
\in\Sym^3(\set{V})
\end{equation}
is a minimum cube-free CP decomposition.  Then $e(\tens{R})\ge1$ and
\[
\DeltaT(\tens{R})\le9e(\tens{R})-3.
\]
\end{proposition}

For the remainder of this section, $\tens{R}$, $\set{V}$, and the displayed
decomposition in~\eqref{eq:cube-free-cpd} satisfy the hypotheses of
\cref{prop:cube-free-residual}.  Put
$r=r(\tens{R}),  d=d(\tens{R})=\dim \set{V}$ and $e=e(\tens{R})=r-d.$  
By \cref{lem:minimum-cpd}, each of the three factor lists spans $\set{V}$.
Symmetry under the three transpositions gives
\begin{equation}
\label{eq:alternating-identities}
\sum_i(\vect{a}_i\wedge \vect{b}_i)\otimes \vect{c}_i=\tens{0},\ \
\sum_i(\vect{b}_i\wedge \vect{c}_i)\otimes \vect{a}_i=\tens{0},\ \
\sum_i(\vect{c}_i\wedge \vect{a}_i)\otimes \vect{b}_i=\tens{0}.
\end{equation}

For a nonzero decomposable bivector $\vect{p}=\vect{u}\wedge \vect{v}$, its support is the
two-dimensional subspace
$\operatorname{supp}(\vect{p})=\Span\{\vect{u},\vect{v}\}.$ 
This subspace is independent of the chosen decomposition of $\vect{p}$.
Equivalently, $\operatorname{supp}(\vect{p})$ is the image of the contraction map
$\set{V}^*\to\set{V}$, $\alpha\mapsto\alpha\mathbin{\lrcorner}\vect{p}$;
in particular, $\vect{p}\in\bigwedge^2\set{H}$ implies
$\operatorname{supp}(\vect{p})\subseteq\set{H}$ for every subspace
$\set{H}\subseteq\set{V}$. 
For $xy\in\{ab,bc,ca\}$, define the observed wedge family, its span,
and its support hull by
\[
\begin{aligned}
\set{P}_{ab}&=\{\vect{a}_i\wedge \vect{b}_i:1\le i\le r\},\ \
\set{L}_{ab}=\Span\set{P}_{ab},\ \
\set{H}_{ab}=\Span\{\vect{a}_i,\vect{b}_i:\vect{a}_i\wedge \vect{b}_i\ne\vect{0}\},\\[2pt]
\set{P}_{bc}&=\{\vect{b}_i\wedge \vect{c}_i:1\le i\le r\},\ \ 
\set{L}_{bc}=\Span\set{P}_{bc},\ \
\set{H}_{bc}=\Span\{\vect{b}_i,\vect{c}_i:\vect{b}_i\wedge \vect{c}_i\ne\vect{0}\},\\[2pt]
\set{P}_{ca}&=\{\vect{c}_i\wedge \vect{a}_i:1\le i\le r\},\ \ 
\set{L}_{ca}=\Span\set{P}_{ca},\ \
\set{H}_{ca}=\Span\{\vect{c}_i,\vect{a}_i:\vect{c}_i\wedge \vect{a}_i\ne\vect{0}\}.
\end{aligned}
\]
Write
$\ell_{xy}=\dim \set{L}_{xy}, 
h_{xy}=\dim \set{H}_{xy}.$ 
For a finite family $\set{P}$ of decomposable bivectors, let $\nu(\set{P})$ be the
number of distinct one-dimensional subspaces $\RR\vect{p}$ represented by
nonzero $\vect{p}\in\set{P}$; zero bivectors are ignored.

\subsection{Alternating equations and wedge-support bounds}

The first support estimate converts each alternating symmetry equation
into a matrix-kernel bound and then spans every supporting plane by a
basis of observed bivectors.

\begin{lemma}[Wedge-span and support bounds]
\label{lem:wedge-support-bounds}
For every $xy\in\{ab,bc,ca\}$,
\[
\ell_{xy}\le e,
\ \ 
h_{xy}\le2\ell_{xy}.
\]
\end{lemma}

\begin{proof}
Consider the pair $ab$.  Choose bases of $\set{V}$ and $\bigwedge^2\set{V}$, let
$\matr{C}=(\vect{c}_1\ \cdots\ \vect{c}_r)$, and let $\matr{W}$ be the matrix whose $i$th column
is the coordinate vector of $\vect{a}_i\wedge \vect{b}_i$.  By the first identity
in~\eqref{eq:alternating-identities}, we have
$\matr{C}\matr{W}^{\mathsf T}=\matr{0}$.  Since the $\vect{c}_i$ span
$\set{V}$, we have $\dim\ker\matr{C}=r-d=e$; consequently,
$$\ell_{ab}=\rank\matr{W}\le\dim\ker\matr{C}=e.$$  
Choose $\ell_{ab}$ observed nonzero bivectors forming a basis of
$\set{L}_{ab}$, and let $\set{H}_0$ be the sum of their support planes.  Then
$\set{L}_{ab}\subseteq\bigwedge^2\set{H}_0$ and
$\dim\set{H}_0\le2\ell_{ab}$.
Every observed bivector belongs to $\bigwedge^2\set{H}_0$, so its support is
contained in $\set{H}_0$.  Thus $\set{H}_{ab}=\set{H}_0$ and
$h_{ab}\le2\ell_{ab}$.  The other two pairs are identical.
\end{proof}

\subsection{Bivector support saving and direction-group partialization}

The next lemma is the exterior-algebra saving in the non-partial branch.
It uses the decomposability identity $\vect{p}\wedge\vect{p}=\vect{0}$ to rule out
full two-dimensional support for each independent bivector direction.

\begin{lemma}[Bivector support saving]
\label{lem:plucker-support-saving}
Let
$\set{P}=\{\vect{p}_\alpha=\vect{a}_\alpha\wedge \vect{b}_\alpha:\alpha\in\set{I}\}
\subseteq\bigwedge^2\set{V}$ 
be a family of nonzero decomposable bivectors.  Put
$\set{L}=\Span\set{P},
m=\dim \set{L}$ and 
$\set{H}(\set{P})=\Span\{\vect{a}_\alpha,\vect{b}_\alpha:\alpha\in\set{I}\}.$ 
If $\nu(\set{P})\ge m+1$, then
\[
\dim \set{H}(\set{P})\le2m-1.
\]
\end{lemma}

\begin{proof}
Choose $m$ members $\vect{p}_1,\ldots,\vect{p}_m$ forming a basis of $\set{L}$, and put
$\set{E}_j=\operatorname{supp}(\vect{p}_j)$.  Every observed bivector belongs to
$\set{L}\subseteq\bigwedge^2(\set{E}_1+\cdots+\set{E}_m)$,
so $\set{H}(\set{P})=\set{E}_1+\cdots+\set{E}_m$ and $\dim \set{H}(\set{P})\le2m$. 
Suppose equality holds.  Then
$\set{H}(\set{P})=\set{E}_1\oplus\cdots\oplus\set{E}_m$.  For
$\vect{p}=\sum_{j=1}^m x_j \vect{p}_j\in\set{L}$, we have
\[
\vect{p}\wedge \vect{p}
=2\sum_{1\le j<k\le m}x_j x_k\,\vect{p}_j\wedge \vect{p}_k.
\]
The nonzero four-vectors $\vect{p}_j\wedge \vect{p}_k$ lie in distinct multidegree
components relative to $\bigoplus_j \set{E}_j$, and hence are linearly
independent.  Since a nonzero bivector is decomposable exactly when its
exterior square vanishes, decomposability forces $x_jx_k=0$ for all
$j\ne k$.  Thus the only decomposable one-dimensional subspaces represented
in $\set{L}$ are $\RR\vect{p}_1,\ldots,\RR\vect{p}_m$, contrary to
$\nu(\set{P})\ge m+1$.
\end{proof}

A \emph{partial CP decomposition} is a decomposition of the form
\[
\tens{R}=\sum_j \vect{u}_j\otimes \vect{v}_j^{\otimes2}.
\]

\begin{lemma}[Minimal sublist replacement]
\label{lem:minimal-sublist-replacement}
Let
\(
 \tens{T}=\sum_{i\in\set{I}}
 \vect{a}_i\otimes\vect{b}_i\otimes\vect{c}_i
\)
be a minimum CP decomposition and, for \(\set{J}\subseteq\set{I}\), put
\(
 \tens{T}_{\set{J}}=\sum_{i\in\set{J}}
 \vect{a}_i\otimes\vect{b}_i\otimes\vect{c}_i.
\)
Every CP decomposition of \(\tens{T}_{\set{J}}\) has at least \(|\set{J}|\) terms.
Consequently, replacing \(\tens{T}_{\set{J}}\) by a decomposition with at most
\(|\set{J}|\) terms preserves both the length and the minimality of the full
decomposition.
\end{lemma}

\begin{proof}
If \(r(\tens{T}_{\set{J}})<|\set{J}|\), adjoining the unchanged terms indexed by
\(\set{I}\setminus\set{J}\) gives
\(r(\tens{T})\le r(\tens{T}_{\set{J}})+|\set{I}\setminus\set{J}|<|\set{I}|\),
contrary to minimality.  Thus \(r(\tens{T}_{\set{J}})\ge |\set{J}|\), and the last
assertion follows.
\end{proof}

Grouping proportional wedge directions and diagonalizing the resulting
binary quadratic tensors gives the following partialization criterion.

\begin{lemma}[Direction-group partialization]
\label{lem:direction-group-partialization}
Under the standing hypotheses, suppose that
$\nu(\set{P}_{bc})=\dim \set{L}_{bc}.$ 
Then $\tens{R}$ admits a minimum partial CP decomposition of length $r$:
\[
\tens{R}=\sum_{j=1}^{r}\vect{u}_j\otimes \vect{v}_j^{\otimes2}.
\]
The analogous statement holds for either of the other two pairs of
modes.
\end{lemma}

\begin{proof}
Put $m=\dim \set{L}_{bc}$.  If $m=0$, then $\vect{b}_i\wedge \vect{c}_i=\vect{0}$ for every $i$,
and the displayed decomposition is already partial. 
Assume $m\ge1$.  Choose representatives
$\vect{\Omega}_1,\ldots,\vect{\Omega}_m$ of all the one-dimensional subspaces
represented by the nonzero elements of $\set{P}_{bc}$.  These $m$ bivectors span the $m$-dimensional
space $\set{L}_{bc}$, and hence form a basis. Partition the indices with
nonzero $bc$-wedge into
\[
\set{I}_k=\{i:\vect{b}_i\wedge \vect{c}_i=\lambda_i\vect{\Omega}_k
      \text{ for some }\lambda_i\in\RR^\times\},
\ \  1\le k\le m,
\]
and put $\set{I}_0=\{i:\vect{b}_i\wedge \vect{c}_i=\vect{0}\}$.
The alternating identity in the last two modes gives
$\tens{0}
=\sum_{k=1}^m
\vect{\Omega}_k\otimes
\left(\sum_{i\in \set{I}_k}\lambda_i \vect{a}_i\right).$ 
Since the $\vect{\Omega}_k$ are linearly independent, we have 
\begin{equation}
\label{eq:direction-group-relation}
\sum_{i\in \set{I}_k}\lambda_i \vect{a}_i=\vect{0}
\ \ (1\le k\le m).
\end{equation} 
Fix $k$, set $\set{E}_k=\operatorname{supp}(\vect{\Omega}_k)$, and define
$\tens{R}_k=\sum_{i\in \set{I}_k}\vect{a}_i\otimes \vect{b}_i\otimes \vect{c}_i.$ 
For $i\in \set{I}_k$, both $\vect{b}_i$ and $\vect{c}_i$ belong to the two-dimensional
space $\set{E}_k$.  Moreover,~\eqref{eq:direction-group-relation} gives
$\sum_{i\in\set{I}_k}\vect{a}_i\otimes
(\vect{b}_i\wedge\vect{c}_i)=\tens{0}$.  Thus $\tens{R}_k$ is symmetric in its
last two modes and
$\tens{R}_k\in\Span\{\vect{a}_i:i\in\set{I}_k\}
\otimes\Sym^2(\set{E}_k)$.

No $\set{I}_k$ has cardinality one,
because~\eqref{eq:direction-group-relation} would then force the nonzero factor
$\vect{a}_i$ to vanish.  If $|\set{I}_k|=2$, the two first factors are proportional,
so $\tens{R}_k=\vect{u}_k\otimes \tens{Q}_k$ for some
$\tens{Q}_k\in\Sym^2(\set{E}_k)$.  The real spectral theorem gives
$\tens{Q}_k=\sum_{j=1}^{q_k}\mu_{kj}\vect{v}_{kj}^{\otimes2}$ for
$q_k\le2$, and therefore
$$\tens{R}_k=\sum_{j=1}^{q_k}(\mu_{kj}\vect{u}_k)
\otimes\vect{v}_{kj}^{\otimes2}.$$
The coefficients $\mu_{kj}$, including negative ones, are absorbed
into the first factors; no positivity assumption is used. 
If $|\set{I}_k|\ge3$, choose a basis $\vect{x}_k,\vect{y}_k$ of $\set{E}_k$.  The three squares
$\vect{x}_k^{\otimes2},\ 
\vect{y}_k^{\otimes2}$ and $(\vect{x}_k+\vect{y}_k)^{\otimes2}$ 
form a basis of $\Sym^2(\set{E}_k)$.  Hence $\tens{R}_k$ has a
partial CP decomposition with at most three terms, and therefore with
at most $|\set{I}_k|$ terms.

Finally, if $i\in \set{I}_0$, then $\vect{c}_i=\mu_i \vect{b}_i$ for some
$\mu_i\in\RR^\times$, and
$\vect{a}_i\otimes \vect{b}_i\otimes \vect{c}_i
=(\mu_i \vect{a}_i)\otimes \vect{b}_i^{\otimes2}.$ 
Concatenating the decompositions of the $\tens{R}_k$ with these terms gives a
partial CP decomposition of $\tens{R}$ with at most $r$ nonzero terms.  Since
each replacement is applied to a disjoint sublist,
\cref{lem:minimal-sublist-replacement} shows that no group can be shortened.
Hence the concatenated decomposition has exactly \(r=r(\tens{R})\) terms and
is minimum.
\end{proof}

\subsection{Partial and non-partial residuals}

In the partial branch, a single alternating equation bounds both the
essential dimension and the length of the cube-free residual.

\begin{lemma}[Partial residual bound]
\label{lem:partial-residual-bound}
Let $\tens{R}$ satisfy the hypotheses of
\cref{prop:cube-free-residual} and admit a minimum partial CP
decomposition.  Then $\tens{R}=\tens{M}'+\tens{S}$, where
$\tens{M}'$ is a sum of displayed cubes and either
$\tens{S}=\tens{0}$ or
$\tens{S}$ has a concise, minimum, cube-free partial decomposition.
Moreover, we have 
$e(\tens{S})\le e(\tens{R})$ and $\DeltaT(\tens{R})\le\DeltaT(\tens{S}).$ 
If $\tens{S}\ne\tens{0}$, then
$d(\tens{S})\le2e(\tens{S})$ and $r(\tens{S})\le3e(\tens{S}).$ 
In all cases,
$\DeltaT(\tens{R})\le6e(\tens{R}).$ 
\end{lemma}

\begin{proof}
Apply \cref{lem:displayed-cube-deletion} to the minimum partial
decomposition.  This gives the stated decomposition and the first two
inequalities.  Assume $\tens{S}\ne\tens{0}$, replace its ambient space by
$\Ess(\tens{S})$, and write its residual partial decomposition as
$\tens{S}=\sum_i\vect{u}_i\otimes\vect{v}_i^{\otimes2}$.
By symmetry in the first two modes, we have
$\sum_i(\vect{u}_i\wedge\vect{v}_i)\otimes\vect{v}_i=\tens{0}$.
The relation space of the repeated-factor list $(\vect{v}_i)$ has dimension
$e(\tens{S})$.  The same matrix argument as in
\cref{lem:wedge-support-bounds} therefore gives
$$\dim\Span\{\vect{u}_i\wedge \vect{v}_i\}\le e(\tens{S}).$$
Every term in this displayed cube-free partial decomposition is a noncube, so
$\vect{u}_i\wedge \vect{v}_i\ne\vect{0}$.  The
support hull of this wedge family therefore contains all the $\vect{u}_i$ and
$\vect{v}_i$.  Since both factor lists span $\Ess(\tens{S})$, we have
$d(\tens{S})\le2e(\tens{S})$ and
$r(\tens{S})=d(\tens{S})+e(\tens{S})\le3e(\tens{S})$.
Since $\tens{S}$ is symmetric, symmetrizing its displayed partial
decomposition leaves $\tens{S}$ unchanged.  Each term
$\vect{u}_i\otimes \vect{v}_i^{\otimes2}$ costs at most three real cubes, so
\[
\DeltaT(\tens{S})
\le3r(\tens{S})-r(\tens{S})
=2r(\tens{S})
\le6e(\tens{S})
\le6e(\tens{R}).
\]
The case $\tens{S}=\tens{0}$ is immediate.
\end{proof}

If partialization is impossible, each wedge family represents one more
one-dimensional subspace than its span dimension.  The preceding
bivector lemma then saves one support dimension in every mode pair.

\begin{lemma}[Strict support bounds in the non-partial branch]
\label{lem:nonpartial-support-bound}
Suppose that $\tens{R}$ admits no minimum partial CP decomposition of length
$r$.  Then, for every $xy\in\{ab,bc,ca\}$, we have 
$\nu(\set{P}_{xy})\ge\ell_{xy}+1$ and 
$h_{xy}\le2\ell_{xy}-1.$  
Consequently,
\[
h_{ab}+h_{bc}+h_{ca}
\le2(\ell_{ab}+\ell_{bc}+\ell_{ca})-3
\le6e-3.
\]
\end{lemma}

\begin{proof}
Fix $xy\in\{ab,bc,ca\}$ and put $m=\ell_{xy}$.  The nonzero observed
bivectors span $\set{L}_{xy}$, so $\nu(\set{P}_{xy})\ge m$.  Equality would give a
minimum partial CP decomposition of length $r$ by
\cref{lem:direction-group-partialization}, after permuting modes if
necessary.  Since $\tens{R}$ is symmetric, a repeated-factor decomposition in
any pair of modes can be cyclically permuted into one with the last two
factors repeated.  This contradicts the hypothesis.  Therefore
$\nu(\set{P}_{xy})\ge m+1.$ 
Applying \cref{lem:plucker-support-saving} to the nonzero observed
bivectors gives
$h_{xy}\le2m-1.$ 
Summing over the three pairs and using
\cref{lem:wedge-support-bounds} gives the desired inequality.
\end{proof}

\subsection{Cross-term count and completion}

Define
\[
\begin{aligned}
\set{I}_a&=\{i:\vect{b}_i\wedge \vect{c}_i=\vect{0}\},\ \
\set{I}_b=\{i:\vect{c}_i\wedge \vect{a}_i=\vect{0}\},\ \
\set{I}_c=\{i:\vect{a}_i\wedge \vect{b}_i=\vect{0}\},\\
\set{I}_x&=\{i:\vect{a}_i\wedge \vect{b}_i\ne\vect{0},\ \vect{b}_i\wedge \vect{c}_i\ne\vect{0},\
                \vect{c}_i\wedge \vect{a}_i\ne\vect{0}\}.
\end{aligned}
\]
Cube-freeness implies that a summand cannot have two vanishing pairwise
wedges.  Thus these four sets are disjoint and exhaust
$\{1,\ldots,r\}$.
The next count uses the spanning property of each factor list modulo
the three wedge-support hulls.

\begin{lemma}[Direct cross-term count]
\label{lem:direct-cross-count}
One has
\[
|\set{I}_x|\le e-2d+h_{ab}+h_{bc}+h_{ca}.
\]
\end{lemma}

\begin{proof}
If $i\notin \set{I}_a$, then $\vect{b}_i\wedge \vect{c}_i\ne\vect{0}$, and hence
$\vect{b}_i\in \set{H}_{bc}$.  Since the complete second-factor list spans $\set{V}$, the
images of the $\vect{b}_i$ with $i\in \set{I}_a$ span $\set{V}/\set{H}_{bc}$.  Therefore, we have
$|\set{I}_a|\ge d-h_{bc}.$ 
Similarly, using the third-factor list for $\set{I}_b$ and the first-factor
list for $\set{I}_c$, one obtains
$|\set{I}_b|\ge d-h_{ca}$ and $|\set{I}_c|\ge d-h_{ab}.$ 
Since the four index sets form a disjoint partition, we obtain
$$r-|\set{I}_x|=|\set{I}_a|+|\set{I}_b|+|\set{I}_c|
\ge3d-(h_{ab}+h_{bc}+h_{ca}).$$
Using $r=d+e$ and rearranging proves the claim.
\end{proof}

\begin{proof}[Proof of \cref{prop:cube-free-residual}]
First, $e\ge1$.  Indeed, if $e=0$, each factor list is linearly
independent.  The first identity
in~\eqref{eq:alternating-identities} would give
$\vect{a}_i\wedge \vect{b}_i=\vect{0}$ termwise, and a second identity would give
$\vect{b}_i\wedge \vect{c}_i=\vect{0}$ termwise.  Every summand would then be a displayed
cube, contrary to cube-freeness.
Suppose first that $\tens{R}$ admits a minimum partial CP decomposition of
length $r$.  \cref{lem:partial-residual-bound} gives
$\DeltaT(\tens{R})\le6e\le9e-3$, since $e\ge1$.

Suppose now that no such partial decomposition exists.  A summand
indexed by $\set{I}_a\cup \set{I}_b\cup \set{I}_c$ has two proportional factors, so its
full symmetrization has real symmetric rank at most three.  A summand
indexed by $\set{I}_x$ has arbitrary factors, and its full symmetrization has
real symmetric rank at most four.  Since $\tens{R}$ is symmetric, we have
$s(\tens{R})\le3(r-|\set{I}_x|)+4|\set{I}_x|=3r+|\set{I}_x|$.
Consequently, we have 
$\DeltaT(\tens{R})\le2r+|\set{I}_x|.$
By \cref{lem:direct-cross-count,lem:nonpartial-support-bound},
it follows that 
\[
\begin{aligned}
\DeltaT(\tens{R})
&\le2(d+e)+e-2d+h_{ab}+h_{bc}+h_{ca}\\
&=3e+h_{ab}+h_{bc}+h_{ca}
\le3e+2(\ell_{ab}+\ell_{bc}+\ell_{ca})-3 
\le9e-3.
\end{aligned}
\]
This proves the proposition.
\end{proof}

\begin{proof}[Proof of \cref{prop:linear-excess-main}]
If $e(\tens{T})=0$, the factor lists in a minimum CP decomposition are
independent, and the alternating identities force every term to be a
cube.  Thus $\DeltaT(\tens{T})=0$. 
Assume $e(\tens{T})\ge1$, and apply
\cref{lem:displayed-cube-deletion} to write $\tens{T}=\tens{M}+\tens{R}$.  If $\tens{R}=\tens{0}$,
the conclusion is immediate.  Otherwise, \cref{prop:cube-free-residual} gives
$\DeltaT(\tens{T})\le\DeltaT(\tens{R})
\le9e(\tens{R})-3\le9e(\tens{T})-3.$ 
\end{proof}

\subsection{Auxiliary square-covering lemmas}

We record two basis-free square-covering lemmas for the low-excess
argument, including degenerate indefinite forms.  In both lemmas, quadratic forms on
$\set{E}$ are elements of $\Sym^2(\set{E}^*)$, paired naturally with
$\Sym^2(\set{E})$.

\begin{lemma}[Real null-cone determination]
\label{lem:real-null-cone}
Let $\set{E}$ be a real vector space and let $\tens{Q},\tens{P}$ be
quadratic forms on $\set{E}$.
Suppose that $\tens{Q}$ takes both positive and negative values and
$\tens{Q}(\vect{x})=0\Rightarrow\tens{P}(\vect{x})=0$.
Then $\tens{P}=c\tens{Q}$ for some $c\in\RR$.
\end{lemma}

\begin{proof}
% LEAN: real_null_cone_determines_quadratic.
Write $\tens{B}_{\tens{Q}}$ and $\tens{B}_{\tens{P}}$ for the associated symmetric
bilinear forms, and normalize $\vect{b}$ so that $\tens{Q}(\vect{b})=-1$.
For $\tens{Q}(\vect{x})>0$, put
$a=\tens{B}_{\tens{Q}}(\vect{x},\vect{b})$ and
$\sigma=\sqrt{a^2+\tens{Q}(\vect{x})}>0$.  Then
\[
 \tens{Q}\bigl(\vect{x}+(a\pm\sigma)\vect{b}\bigr)=0.
\]
Subtracting the two corresponding equations for $\tens{P}$ gives
$\tens{B}_{\tens{P}}(\vect{x},\vect{b})=-a\tens{P}(\vect{b})$;
substitution gives $\tens{P}(\vect{x})=-\tens{P}(\vect{b})\tens{Q}(\vect{x})$.
Thus $\tens{H}=\tens{P}+\tens{P}(\vect{b})\tens{Q}$ vanishes wherever
$\tens{Q}>0$.  Choose $\vect{u}$ with $\tens{Q}(\vect{u})>0$.
For any $\vect{x}$ and sufficiently large $t$,
$\tens{Q}(\vect{x}\pm t\vect{u})>0$, so
\[
 2\tens{H}(\vect{x})
 =\tens{H}(\vect{x}+t\vect{u})+\tens{H}(\vect{x}-t\vect{u})
   -2t^2\tens{H}(\vect{u})=0.
\]
Hence $\tens{P}=-\tens{P}(\vect{b})\tens{Q}$.
\end{proof}

\begin{lemma}[Square covers of codimension two]
\label{lem:codimension-two-square-cover}
Let $\set{E}$ be a finite-dimensional real vector space,
$N=\dim\Sym^2(\set{E})$, and let
$\set{S}\subseteq\Sym^2(\set{E})$ be a subspace with $\dim\set{S}\le N-2$.
There are $\vect{v}_1,\ldots,\vect{v}_{N-1}\in\set{E}$ such that
\[
 \set{S}\subseteq\Span\{\vect{v}_1^2,\ldots,\vect{v}_{N-1}^2\}.
\]
\end{lemma}

\begin{proof}
% LEAN: square_cover_of_codimension_two; nullSquareSpan_eq_hyperplane.
Choose independent $\tens{P},\tens{R}\in\set{S}^{\perp}$ and
$\vect{a}$ with $\tens{P}(\vect{a})\ne0$.
Replacing $\tens{R}$ by
$\tens{R}-\tens{R}(\vect{a})\tens{P}/\tens{P}(\vect{a})$,
we may assume $\tens{R}(\vect{a})=0$ while $\tens{R}\ne\tens{0}$.
Choose $\vect{b}$ with $\tens{R}(\vect{b})\ne0$ and put
\[
 \tens{Q}=\tens{P}(\vect{a})
 \left(\tens{P}
 -\frac{\tens{P}(\vect{a})+\tens{P}(\vect{b})}{\tens{R}(\vect{b})}\tens{R}\right)
 \in\set{S}^{\perp}.
\]
Then $\tens{Q}(\vect{a})=\tens{P}(\vect{a})^2>0$ and
$\tens{Q}(\vect{b})=-\tens{P}(\vect{a})^2<0$.
By \cref{lem:real-null-cone} and duality, we have
$\Span\{\vect{v}^2:\tens{Q}(\vect{v})=0\}=\tens{Q}^{\perp}.$ 
This hyperplane has dimension $N-1$ and contains $\set{S}$.
Selecting a basis of null squares proves the assertion.
\end{proof}

\section{Zero gap through excess three}
\label{sec:proof-small-excess}

This section proves \cref{prop:excess-three-zero}.  Unless stated
otherwise, $\tens{T}$ is a concise real symmetric cubic with a fixed
minimum CP decomposition, and
\(r,d,e\) denote
\(r(\tens{T}),d(\tens{T}),r(\tens{T})-d(\tens{T})\).
We use the classical equalities
$\Mthree(1)=1$, $\Mthree(2)=3$, and $\Mthree(3)=5$
\cite{reznick2013length,banchi2015ternary}; the inequality
$\Mthree(4)\leq7$ is proved in \cref{lem:four-variable-seven}.

\subsection{Auxiliary low-dimensional bounds and descent}

\subsubsection{A seven-cube bound in four variables}

We identify $\Sym^3(\set{U})$ with the cubic forms on $\set{U}^*$.
For $\alpha\in\set{U}^*$, let $D_\alpha \tens{T}$ denote the directional derivative;
thus $D_\alpha L^3=3\alpha(L)L^2$ and
$D_\alpha\tens{T}=3(\alpha\mathbin{\lrcorner}\tens{T})$ under the
contraction normalization used below.

\begin{lemma}[A real polar of rank at most two]
\label{lem:quaternary-low-rank-polar}
Every $\tens{T}\in\Sym^3(\RR^4)$ admits a nonzero
$\alpha\in(\RR^4)^*$ such that
\[
 \rank(D_\alpha \tens{T})\leq2.
\]
\end{lemma}

\begin{proof}
Assume first that $\tens{T}$ is general.  By Sylvester's pentahedral theorem
\cite[Theorem~2.1]{seigal2020ranks}, its complex Waring decomposition
$\tens{T}=L_1^3+\cdots+L_5^3$ is unique.  Complex conjugation therefore permutes the five
one-dimensional subspaces $\mathbb{C}L_i$.  Every involution on a five-element set has a conjugation-invariant
three-element subset: take three fixed points when they exist, and otherwise
take one fixed point together with a conjugate pair.  If $\set{S}$ is such a
subset, then the common annihilator
$$\set{K}=(\Span_{\mathbb C}\{L_i:i\in\set{S}\})^\perp$$ 
is a nonzero conjugation-invariant complex subspace.  Hence $\set{K}$ is the
complexification of its real fixed subspace and contains a nonzero real
covector $\alpha$.  The real quadratic $D_\alpha\tens{T}$ has complex
matrix rank at most two, because over $\mathbb C$ it is the sum of the remaining
two squares.  Scalar extension does not change the rank of a real matrix, so
$\rank_{\RR}(D_\alpha\tens{T})=\rank_{\mathbb C}(D_\alpha\tens{T})\le2$.

The pentahedral locus is a nonempty Zariski-open subset defined over $\RR$.
Its complement is a proper real algebraic subset and therefore has empty
Euclidean interior, so the real points of this locus are Euclidean dense.
For arbitrary $\tens{T}$, choose general real cubics
$\tens{T}_n\to \tens{T}$ and unit real covectors $\alpha_n$ as above.  After passing to a subsequence,
$\alpha_n\to\alpha$ with $\|\alpha\|=1$.  All $3\times3$ minors of
$D_{\alpha_n}\tens{T}_n$ vanish.  Passing to the limit gives
$\rank(D_\alpha \tens{T})\le2$.
\end{proof}

We use repeatedly the real polarization identities
\begin{align*}
 x^2M&=\frac{(x+M)^3-(x-M)^3-2M^3}{6},\\
 24UVW&=(U+V+W)^3+(U-V-W)^3+(-U+V-W)^3+(-U-V+W)^3.
\end{align*}
Throughout, a real scalar multiple of a cube counts as one real cube, since
the scalar can be absorbed into the linear form by taking its real cube root.

\begin{proposition}[Four-variable real cubics]
\label{lem:four-variable-seven}
Every real cubic in at most four variables is a sum of at most seven
real cubes.  Equivalently,
\[
 \Mthree(4)\le7.
\]
\end{proposition}

\begin{proof}
Choose $\alpha$ as in \cref{lem:quaternary-low-rank-polar}, put
$\tens{Q}=D_\alpha \tens{T}$, and write $h=\rank\tens{Q}\le2$.
If $h=0$, then $\tens{T}$ is a ternary cubic and has real Waring rank at most five
by \cite{banchi2015ternary}.  If
$\Ess(\tens{Q})\not\subseteq\ker\alpha$, diagonalization on the
support of $\tens{Q}$, followed by a real orthogonal or hyperbolic change of basis,
gives $\tens{Q}=\sum_{i=1}^{h}c_iL_i^2$ with $\alpha(L_i)\ne0$.
Indeed, the case $h=1$ is immediate.  If $h=2$, put
$\set{W}=\Ess(\tens{Q})$, and let $b_{\tens{Q}}$ be the nondegenerate
bilinear form on $\set{W}$ corresponding to the inverse of
$\tens{Q}^{\sharp}:\set{W}^*\to\set{W}$.  Choose
$\vect{v}\in\set{W}$ outside the finite union of the $b_{\tens{Q}}$-isotropic lines,
$\ker(\alpha|_{\set{W}})$, and
$(\ker(\alpha|_{\set{W}}))^{\perp_{b_{\tens{Q}}}}$.  If $\vect{w}$ spans
$\vect{v}^{\perp_{b_{\tens{Q}}}}$, then
$\alpha(\vect{v})\alpha(\vect{w})\ne0$, and
$(\vect{v},\vect{w})$ is a $b_{\tens{Q}}$-orthogonal basis, which diagonalizes
$\tens{Q}$ in the required form.
Consequently, we obtain
$$D_\alpha\!\left(\tens{T}-\sum_{i=1}^{h}
 c_iL_i^3/(3\alpha(L_i))\right)=\tens{0},$$
and the residual is ternary.  Thus $s(\tens{T})\le h+5\le7$.

It remains to treat the tangent case
$\Ess(\tens{Q})\subseteq\ker\alpha$.

\smallskip
\noindent\emph{Rank one.}
After replacing $\alpha$ by $-\alpha$ if necessary and rescaling $x$, choose
coordinates such that $\alpha=D_t$ and $\tens{Q}=x^2$.
Then we have $\tens{T}=tx^2+\tens{R}(x,y,z)$ and
$\tens{R}\equiv P(y,z)+xR(y,z)\pmod{x^2}$.
If $P\ne0$, choose four distinct real binary linear forms $L_i$ such that
$P=\sum_{i=1}^{4}c_iL_i^3$ and $c_i\ne0$.
To see that such a choice exists, fix coordinates on $\Sym^3(\RR^2)$.
For each $i$, the determinant with columns $P$ and $L_j^3$, $j\ne i$,
is a nonzero polynomial in the coefficients of the $L_j$: three cubes
can be chosen so that they and $P$ form a basis.  The product of these
four determinants and the pairwise $2\times2$ determinants of the $L_i$
is therefore a nonzero real polynomial.  Choose a real tuple where it
does not vanish.  The Vandermonde determinant then shows that the four
cubes $L_i^3$ form a basis, and the four coordinates $c_i$ of $P$ are nonzero.
Moreover, the squares of any three distinct binary linear forms are
independent, so the $L_i^2$ span $\Sym^2(\RR^2)$.  Write
$R=\sum_i d_iL_i^2$.  Then
\[
 \sum_{i=1}^{4}c_i
 \left(L_i+\frac{d_i}{3c_i}x\right)^3
 \equiv P+xR\pmod{x^2}.
\]
If $P=0$, a spectral decomposition of $R$ and
$(L+qx/3)^3-L^3\equiv qxL^2\pmod{x^2}$ give the same conclusion with
at most four cubes.  The remaining term is $x^2M$ and costs three cubes.

\smallskip
\noindent\emph{Rank two.}
Choose coordinates such that $\alpha=D_t$.  Then we have
$$\tens{T}=t\tens{Q}(x,y)+P(x,y)+zR(x,y)+z^2L(x,y)+cz^3,$$
where $\tens{Q}$ is a nondegenerate binary quadratic form.  Put
$\tens{G}=P+zR+z^2L+cz^3$. 
Suppose first that $c\ne0$.  If $\tens{Q}$ is definite, take $\tens{Q}=x^2+y^2$ and
$n^3=c$.  Subtracting $(nz+L/(3n^2))^3$ eliminates the $z^3$ and
$z^2L$ terms.  An orthogonal change of $x,y$
diagonalizes the remaining coefficient of $z$ while preserving $\tens{Q}$.
The residual is then $u^2M+v^2N$, which costs at most six cubes; together
with the subtracted cube, this gives at most seven cubes.

If $\tens{Q}$ is split, take $\tens{Q}=xy$.  We claim that every
$B(x,z)=z^3+3mxz^2+3nx^2z+px^3$
has an expression
\begin{equation}
\label{eq:quaternary-fixed-three-weights}
 B=2(z+a_1x)^3+2(z+a_2x)^3-3(z+a_3x)^3
\end{equation}
% LEAN: exists_fixed_three_weight_decomposition.
with $a_1,a_2,a_3\in\RR$.  Put
$V=n-m^2$, $S=p-3mn+2m^3$, and
$h(\tau)=8\tau^3/9+3V\tau$.
There is a real solution of $h(\tau)=S$ satisfying
$V/4+\tau^2/3\ge0$.  This is immediate for $V\ge0$; for $V<0$, put
$T_0=(-3V/4)^{1/2}$ and use
$h(T_0)=7VT_0/3<0<h(-T_0)$ together with the behavior at infinity.  With
\[
 \sigma=\left(\frac V4+\frac{\tau^2}{3}\right)^{1/2},\ \ 
 a_{1,2}=m+\tau\pm\sigma,\ \  a_3=m+\frac43\tau,
\]
the four weighted moments with weights $(2,2,-3)$ are $1,m,n,p$,
which proves \eqref{eq:quaternary-fixed-three-weights}. 
Apply \eqref{eq:quaternary-fixed-three-weights} to
$\tens{G}(x,0,z)/c$ and $\tens{G}(0,y,z)/c$.
Because the weights are the same,
there are real $K_i=z+a_ix+b_iy$ such that
$\tens{G}-c(2K_1^3+2K_2^3-3K_3^3)$ vanishes on both $x=0$ and $y=0$,
and hence is divisible by $xy$.
Thus $\tens{T}$ is three cubes plus a product of three linear forms, and costs at
most seven cubes.

Now assume $c=0$ and $L\ne0$.  If $L\nmid\tens{Q}$, multiplication by $L$
induces an isomorphism
$\langle x,y\rangle\to\Sym^2\langle x,y\rangle/\RR\,\tens{Q}$.
Choose a linear form $H$ such that $R+2HL=k\tens{Q}$.  The changes
$z=z'+H$ and $t'=t+kz'$ transform $\tens{T}$ into
$$\tens{T}=t'\tens{Q}+z'^2L+P'(x,y).$$ 
The polar $D_{z'}\tens{T}=2z'L$ 
is transverse of rank two, so the first part of
the proof applies. 
If $L\mid\tens{Q}$, take $\tens{Q}=xy$ and $L=x$.  Absorbing all
multiples of $xy$ into $t$ gives
$$\tens{T}=txy+xz^2+ax^2z+\beta y^2z+A_0x^3+B_0y^3.$$ 
For any real $v$, the invertible triangular change
\[
 z=z'-ax/2+vy,\ \ 
 t'=t+2vz'+(v^2-a\beta/2)y
\]
gives
$\tens{T}=t'xy+xz'^2+\beta y^2z'+Ax^3+By^3$, where
$A=A_0-a^2/4$ and $B=B_0+\beta v$.
If $\beta=0$, the $z'$-polar is transverse.  If $\beta\ne0$, take
$v=-B_0/\beta$, so $B=0$, and relabel $t',z'$ as $t,z$.
% LEAN: exists_exceptional_tangent_parameters/completion.
Put $m=12\operatorname{sgn}(\beta)(|A|+1)$ and
$\lambda=1/(3m)$.  Then $\beta m>0$ and $m^2/12>A$, so
$\beta/\lambda>0$ and $1/(108\lambda^2)>A$.  Set
\[
 \tau=-\operatorname{sgn}(\lambda)
       \sqrt{\frac1{108\lambda^2}-A},\quad
 \sigma=\sqrt{\frac{m^2}{4}-m\tau},\quad
 \eta=\sqrt{\frac\beta{6\lambda}}.
\]
These are real: $m\tau\le0$ and $\beta/\lambda>0$.  In particular,
$3\lambda m=1$, $\tau^2=m^2/12-A$,
$\sigma^2=m^2/4-m\tau$, and $6\lambda\eta^2=\beta$.
For
\[
 K_1=z+(\tau+\sigma)x+\eta y,\quad
 K_2=z+(\tau-\sigma)x-\eta y,\quad
 K_3=z+(\tau-m/2)x
\]
direct expansion gives the exact remainder
\[
\begin{aligned}
 \lambda(K_1^3+K_2^3-2K_3^3)-(xz^2+\beta y^2z+Ax^3) =xy(12\lambda\sigma\eta z+12\lambda\sigma\tau\eta x
                +6\lambda\tau\eta^2y).
\end{aligned}
\]
Again $\tens{T}$ is three cubes plus a product of three linear forms.

Finally suppose that $c=L=0$.  Then
$\tens{T}=t\tens{Q}(x,y)+zR(x,y)+P(x,y)$.
If $\tens{Q},R$ are dependent, the form is ternary after a linear change.
Otherwise their real pencil has exactly one of the three normal forms
\[
 \langle x^2,y^2\rangle,\ \ 
 \langle x^2,xy\rangle,\ \ 
 \langle x^2-y^2,2xy\rangle.
\]
Here a change of basis in the pencil is induced by an invertible linear change
of the variables $t,z$.
Indeed, represent $\tens{Q}$ and $R$ by real symmetric $2\times2$ matrices
$\matr{A}$ and $\matr{B}$.  The rank-one locus in
$\mathbb{P}(\Sym^2(\RR^2))$ is a nondegenerate conic and contains no
projective line.  Hence $\det(s\matr{A}+t\matr{B})$ is a nonzero binary
quadratic, so it has two
distinct real zeros, one double real zero, or no real zero.  Sylvester
congruence, followed by a change of basis
in the pencil, gives the three displayed normal forms, respectively.  More
explicitly, in the first case the two singular members have distinct kernels
and reduce to $x^2$ and $y^2$.  In the double-root case, put the unique
rank-one member in the form $x^2$; the double-root condition removes the
$y^2$ coefficient of an independent member, which then reduces modulo $x^2$
to $xy$.  In the last case, the determinant restricts to a definite form on
the pencil.  Since the determinant form on $\Sym^2(\RR^2)$ has signature
$(1,2)$, this restriction is negative definite; its orthogonal complement is
a positive line generated by a definite quadratic form.  Sending that form
to $x^2+y^2$ by congruence identifies its determinant-orthogonal complement
with the trace-free plane $\langle x^2-y^2,2xy\rangle$.
In the first case,
$$\Sym^3\langle x,y\rangle=x^2\langle x,y\rangle+y^2\langle x,y\rangle,$$ 
so an invertible triangular change of $t,z$ absorbs $P$ and writes the result as
$x^2M+y^2N$, which costs six cubes.  In the second case, reduce to
$t'x^2+z'xy+ay^3$, using
$\Sym^3\langle x,y\rangle=x^2\langle x,y\rangle+xy\langle x,y\rangle+\RR y^3$.
The first term costs three cubes.  Write $xz'=u^2-v^2$ and choose
$b,d\ne0$ with $b^2-d^2=3a$; then
\[
 y(u^2-v^2+ay^2)
 =\frac{(by+u)^3+(by-u)^3}{6b}
  -\frac{(dy+v)^3+(dy-v)^3}{6d}.
\]
In the no-real-root case,
$\Sym^3\langle x,y\rangle=(x^2-y^2)\langle x,y\rangle+2xy\langle x,y\rangle$,
so an invertible triangular change of $t,z$ absorbs $P$ and reduces the form to
$\tens{T}=t(x^2-y^2)+2zxy$.
For $\theta_k=2\pi k/7$ one has the real identity
\[
 \tens{T}=\frac4{21}\sum_{k=0}^{6}
 \bigl(x\cos\theta_k+y\sin\theta_k
       +t\cos2\theta_k+z\sin2\theta_k\bigr)^3.
\]
Indeed, put
$A_k=x\cos\theta_k+y\sin\theta_k$ and
$B_k=t\cos2\theta_k+z\sin2\theta_k$.  Seven-point Fourier
orthogonality eliminates all nonzero frequencies in
$(A_k+B_k)^3$, while the constant term of $3A_k^2B_k$ is
$3\tens{T}/4$.  The rank-two tangent case is exhausted, and the
result follows.
\end{proof}

\begin{remark}
Write $M_{3,\mathbb C}(4)$ for the analogous maximum over $\mathbb C$.
Segre's classical result gives $M_{3,\mathbb C}(4)=7$
\cite[Section~97]{segre1942nonsingular}; see also the modern survey
\cite[Section~3.2]{deparis2018maximum}.  Over $\RR$, the only typical ranks
of quaternary cubics are five and six
\cite[Theorem~1.2(2)]{bernardi2018typical}, while the general maximum-rank
theorem of Blekherman and Teitler gives only $\Mthree(4)\le10$
\cite[Theorems~2 and~3]{blekherman2015maximum}.  Seigal exhibited real
quaternary cubics of Waring rank seven, including
$x_1(x_1x_2+x_3^2+x_4^2)$ \cite[Proposition~4.3]{seigal2020ranks}.
Thus \cref{lem:four-variable-seven} is sharp and yields
$\Mthree(4)=7$; to the best of our knowledge, this exact real value has not
previously been recorded.
\end{remark}

\subsubsection{Low wedge rank}

The next square-covering facts supply the only real input required for
partialization.

Here a partial CP decomposition of \(\tens{T}\) is called minimum if its
length is \(r(\tens{T})\).

\begin{lemma}[Square covers in three variables]
\label{lem:square-covers-three}
Let $\set{W}\cong\RR^3$ and $\set{S}\subseteq\Sym^2\set{W}$.
If $\dim \set{S}\le2$, then $\set{S}$ is contained in the span of four real squares.
If $\dim \set{S}\le4$, then $\set{S}$ is contained in the span of five real squares.
\end{lemma}

\begin{proof}
For the first assertion, only $\dim \set{S}=2$ requires attention.  If every
matrix in $\set{S}$ is singular, spectral decompositions of two basis matrices
use at most four squares.  Otherwise choose a basis
$\matr{B},\matr{C}$ of $\set{S}$ with $\matr{B}$ invertible.
The real cubic $\det(t\matr{B}+\matr{C})$ has a real root, so $\set{S}$ contains a
singular matrix $\matr{M}$ independent of $\matr{B}$.  Rank one is immediate.  If $\rank\matr{M}=2$,
write $\set{W}_0=\operatorname{im}\matr{M}$ and choose $\vect{u}\in\set{W}_0$ such that
\[
 q_1(\vect{u})=\langle\vect{u},\matr{M}_{\set{W}_0}^{-1}\vect{u}\rangle\ne0,\ \ 
 q_2(\vect{u})=\langle\vect{u},\matr{B}^{-1}\vect{u}\rangle\ne0.
\]
The first quadratic polynomial is nonzero because
$\matr{M}_{\set{W}_0}$ is invertible.  If the second vanished identically
on $\set{W}_0$, polarization would make the two-plane $\set{W}_0$
totally isotropic for $\matr{B}^{-1}$, so
$\set{W}_0\subseteq\set{W}_0^{\perp_{\matr{B}^{-1}}}$, whereas
$\dim\set{W}_0^{\perp_{\matr{B}^{-1}}}=3-2=1$.  Their product is a nonzero real
polynomial, so such $\vect{u}$ exists.  The rank-one update formula gives
\[
 \rank(\matr{M}-q_1(\vect{u})^{-1}\vect{u}^2)=1,\ \ 
 \rank(\matr{B}-q_2(\vect{u})^{-1}\vect{u}^2)=2.
\]
One common square and spectral decompositions of the two residuals give
four squares. 
The second assertion follows from \cref{lem:codimension-two-square-cover},
since $\dim\Sym^2\set{W}=6$.

\end{proof}

\begin{lemma}[Low-wedge partialization]
\label{lem:low-wedge-partialization}
Let $\tens{T}\in\set{A}\otimes\Sym^2\set{V}$ have a minimum CP decomposition
$\tens{T}=\sum_{i=1}^{r}\vect{a}_i\otimes\vect{b}_i\otimes\vect{c}_i.$ 
If $\rho=\dim\Span\{\vect{b}_i\wedge\vect{c}_i\}\le2$, then $\tens{T}$ has a minimum
partial decomposition of the same length,
\[
 \tens{T}=\sum_{i=1}^{r}\vect{u}_i\otimes\vect{\ell}_i^{\otimes2}.
\]
\end{lemma}

\begin{proof}
Zero-wedge terms are already partial.  Suppose first that $\rho=1$.
Let \(\set{I}\) be the set of nonzero-wedge indices and \(m=|\set{I}|\).  All their
wedges have the same two-dimensional support.  Their alternating relation
makes the subtotal \(\tens{T}_{\set{I}}\) symmetric in its last two modes (the
complementary zero-wedge terms are symmetric termwise), so a
three-square basis of the binary quadratic space gives a partial
decomposition of \(\tens{T}_{\set{I}}\) with at most three terms.
\Cref{lem:minimal-sublist-replacement} gives \(m\le3\).
The alternating relation excludes $m=1$.  If $m=2$, the first factors
are proportional, so the subtotal is $\vect{u}\otimes\tens{Q}$ with
$\rank\tens{Q}\le2$; spectral decomposition gives at most two partial
terms.  If $m=3$, use the three-square basis above.  In both cases,
\cref{lem:minimal-sublist-replacement} forces exactly $m$ terms.

Now let $\rho=2$ and choose two observed wedges as a basis.  If their
supports are disjoint, the identity
$ (a\vect{\omega}_1+b\vect{\omega}_2)^{\wedge2}
=2ab\,\vect{\omega}_1\wedge\vect{\omega}_2 $
shows that the only decomposable directions in their span are the two
basis directions.  For $j=1,2$, let $\set{I}_j$ be the set of indices such that
$\vect{b}_i\wedge\vect{c}_i=\lambda_i\vect{\omega}_j$ for some
$\lambda_i\in\RR^\times$.  Since $\vect{\omega}_1,\vect{\omega}_2$ are independent,
the alternating relation gives
$\sum_{i\in\set{I}_j}\lambda_i\vect{a}_i=\vect{0}$ for $j=1,2$.
Thus each group subtotal is symmetric in its last two modes, and the
preceding rank-one-wedge argument applies separately to the two groups.

If the supports meet, all nonzero-wedge factors lie in a
three-dimensional space \(\set{W}\).  Let \(\set{I}\) contain all nonzero-wedge indices,
put \(m=|\set{I}|\), and let \(\tens{T}_{\set{I}}\) be their subtotal.  The two independent
coordinate rows of the wedge matrix give two independent relations among
the first factors \(\vect{a}_i\), \(i\in\set{I}\).  Hence
\(\dim\Span\{\vect{a}_i:i\in\set{I}\}\le m-2\).  They also make
\(\tens{T}_{\set{I}}\in\set{A}\otimes\Sym^2\set{W}\), because the complementary zero-wedge
terms are symmetric in the last two modes.  Therefore,
$\dim\set{S}_{\tens{T}_{\set{I}}}\le m-2$,
where \(\set{S}_{\tens{T}_{\set{I}}}\subseteq\Sym^2\set{W}\) is the image of the
contraction map 
$$\set{A}^*\to\Sym^2\set{W},\ \alpha\mapsto(\alpha\otimes\operatorname{id}^{\otimes2})(\tens{T}_{\set{I}}).$$
Six standard squares span \(\Sym^2\set{W}\), so
\cref{lem:minimal-sublist-replacement} first gives \(m\le6\).
If \(m=6\), then \(\dim\set{S}_{\tens{T}_{\set{I}}}\le4\), and the five-square
part of \cref{lem:square-covers-three} contradicts minimality.  Since
\(\rho=2\), one has \(m\ge2\), and \(m=2\) is impossible: the two wedges are
independent, so the alternating relation forces both first factors to
vanish.  For \(m=3,4,5\), respectively,
spectral decomposition, the four-square cover, and the five-square cover
give a partial replacement with at most \(m\) terms.
Applying \cref{lem:minimal-sublist-replacement} and adjoining the unchanged
zero-wedge terms yields a minimum partial decomposition of length \(r\).
\end{proof}

\begin{proposition}[The remaining low-wedge partial formats]
\label{lem:partial-low-second-wedge}
Let $\tens{T}$ be a nonzero concise symmetric cubic with a minimum partial
decomposition
\[
 \tens{T}=\sum_{i=1}^{r}\vect{a}_i\otimes\vect{\ell}_i^{\otimes2},\ \ 
 \vect{a}_i\wedge\vect{\ell}_i\ne\vect{0}.
\]
If $\dim\Span\{\vect{a}_i\wedge\vect{\ell}_i\}\le2$, then
\[
 (d(\tens{T}),r(\tens{T}))\in\{(2,3),(3,5),(4,6)\},\ \  s(\tens{T})=r(\tens{T}).
\]
\end{proposition}

\begin{proof}
Symmetry in the first two modes gives
$\sum_{i=1}^{r}(\vect{a}_i\wedge\vect{\ell}_i)\otimes\vect{\ell}_i=\tens{0}$.
Let $\matr{\Omega}$ and $\matr{L}$ be the matrices whose columns are the
coordinate vectors of $\vect{a}_i\wedge\vect{\ell}_i$ and $\vect{\ell}_i$,
respectively.  Then $\matr{\Omega}\matr{L}^{\mathsf T}=\matr{0}$.
Since the decomposition is minimum, its factor lists span $\Ess(\tens{T})$,
so $\rank\matr{L}=d$ and hence
$\dim\Span\{\vect{a}_i\wedge\vect{\ell}_i\}\le r-d$.
If the wedge span has dimension one, all factors lie in its common support
plane $\set{P}$.  The observed support and the factor-list spanning property give
$\Ess(\tens{T})=\set{P}$.  Thus $d=2$ and $r\ge3$, while every real binary cubic
has symmetric rank at most three.  Hence $r=s=3$.

If the wedge span has dimension two, let $\set{P}_1,\set{P}_2$ be the supports of two
basis wedges.  When they meet, their sum is three-dimensional, all factors
lie in $\set{P}_1+\set{P}_2$, and both observed supports lie in
$\Ess(\tens{T})$.  Hence $\Ess(\tens{T})=\set{P}_1+\set{P}_2$, so $d=3$ and
$5\le r\le s\le \Mthree(3)=5$.  When they are disjoint, the
exterior square of $x\vect{\omega}_1+y\vect{\omega}_2$ is
$2xy\vect{\omega}_1\wedge\vect{\omega}_2$, with
$\vect{\omega}_1\wedge\vect{\omega}_2\ne\vect{0}$.  Thus decomposability forces
$xy=0$, so every observed wedge lies on one of the two basis directions.
Accordingly, every summand belongs to either $\set{P}_1^{\otimes3}$ or
$\set{P}_2^{\otimes3}$.  Since both directions occur and each factor list spans
$\Ess(\tens{T})$, one has $\Ess(\tens{T})=\set{P}_1\oplus\set{P}_2$ and hence
$d=4$.  Let $\pi_j:\Ess(\tens{T})\to\set{P}_j$ be the projections associated
with this direct sum, and put $\tens{T}_j=\pi_j^{\otimes3}\tens{T}$.  Symmetry gives
$\tens{T}_j\in\Sym^3\set{P}_j$, and the direction partition gives
$\tens{T}=\tens{T}_1+\tens{T}_2$.  Therefore
$6\le r\le s\le3+3=6$.
\end{proof}

\subsubsection{Saturated descent}

We first isolate the finite-dimensional direction configurations used in
dimensions five and six.

\begin{lemma}[Four bivectors in a three-dimensional wedge space]
\label{lem:four-bivector-configuration}
Let \(\set{L}\subseteq\bigwedge^2\set{V}\) have dimension three.
\begin{enumerate}[label=\textup{(\roman*)}]
\item Let $\vect{\omega}_1,\ldots,\vect{\omega}_4\in\set{L}$ be nonzero
decomposable bivectors such that every three are linearly independent.  Then
$\dim\sum_{i=1}^4\operatorname{supp}(\vect{\omega}_i)\le4$.
Consequently, the support of every nonzero decomposable bivector in $\set{L}$
is contained in a common subspace of dimension at most four.
\item Let $\set{S}$ be a finite set of distinct one-dimensional subspaces of
$\set{L}$ that spans $\set{L}$.  If no four members have the property that
every three span $\set{L}$, then either $|\set{S}|=3$ or all but one member
of $\set{S}$ lie in a common two-dimensional subspace of $\set{L}$.
\end{enumerate}
\end{lemma}

\begin{proof}
% LEAN: four_general_simple_support constructs the four-point circuit and
% its common support; contractionSupport_of_mem_span covers further points.
For \textup{(i)}, put $\set{E}_i=\operatorname{supp}(\vect{\omega}_i)$.
If the first three support planes meet pairwise, put
$\set{H}=\set{E}_1+\set{E}_2+\set{E}_3$.  Then $\dim\set{H}\le4$ and
$\set{L}=\Span\{\vect{\omega}_1,\vect{\omega}_2,\vect{\omega}_3\}
\subseteq\bigwedge^2\set{H}$, so the support of every nonzero decomposable
bivector in $\set{L}$ lies in $\set{H}$.
Otherwise, after relabeling, $\set{E}_1\cap\set{E}_2=\{\vect{0}\}$.
Every coefficient of the unique dependence is nonzero, so rescaling gives
$\vect{\omega}_1+\vect{\omega}_2=\vect{\omega}_4-\vect{\omega}_3$.
Put $\set{H}=\set{E}_1\oplus\set{E}_2$.  For
$\matr{F}_{\vect{\omega}}(\alpha)=\alpha\mathbin{\lrcorner}\vect{\omega}$, one has
$\operatorname{im}\matr{F}_{\vect{\omega}_i}=\set{E}_i$.  Indeed, restriction
$\set{V}^*\to\set{E}_1^*\oplus\set{E}_2^*$ is surjective, while contraction
by the nonzero area form $\vect{\omega}_i$ identifies $\set{E}_i^*$ with
$\set{E}_i$.  The directness of
$\set{E}_1+\set{E}_2$ gives
$\operatorname{im}\matr{F}_{\vect{\omega}_1+\vect{\omega}_2}=\set{H}$.
Hence $\set{H}\subseteq\set{E}_3+\set{E}_4$.  Since $\dim\set{H}=4$ and
$\dim(\set{E}_3+\set{E}_4)\le4$, the two spaces are equal.  Thus all four
supports, and therefore the support of every nonzero decomposable bivector
in their span $\set{L}$, lie in $\set{H}$.

% LEAN: noFourGeneral_iff identifies the determinant condition with triple
% independence; projective_plane_configuration proves the coordinate classification.
For \textup{(ii)}, choose representatives of three members of $\set{S}$ as a
coordinate basis of $\set{L}$.  The hypothesis implies $xyz=0$ for every
additional direction represented by $(x,y,z)$.
If there is no further direction, the assertion holds.  Otherwise, after
permuting coordinates, choose a direction represented by $(a,b,0)$ with
$ab\ne0$.  For a direction represented by $(x,y,z)$ with $z\ne0$, applying
the same three-by-three determinant test shows that $x=0$, $y\ne0$ would
make every three of the four lines through
$(1,0,0),(a,b,0),(0,0,1),(0,y,z)$ independent, while $y=0$, $x\ne0$
would do the same with $(0,1,0)$ in place of $(1,0,0)$.
Since $xyz=0$, one must therefore have $x=y=0$.
Thus every member of $\set{S}$ other than $\RR(0,0,1)$ is represented in
the two-dimensional coordinate subspace $z=0$.
\end{proof}

For \(\tens{T}\in\Sym^3\set{V}\), write
\(\operatorname{pol}_{\tens{T}}:\set{V}^*\to\Sym^2\set{V}\),
\(\alpha\mapsto\alpha\mathbin{\lrcorner}\tens{T}\), and set
\(\set{P}_{\mathrm{pol}}(\tens{T})=
\operatorname{im}(\operatorname{pol}_{\tens{T}})\).
We next record the block principle used in the descent.

\begin{lemma}[Polar blocks]
\label{lem:small-polar-blocks}
Let $\tens{T}\in\Sym^3\set{V}$ be concise and suppose its polar space
has a direct-sum decomposition
\[
 \set{P}_{\mathrm{pol}}(\tens{T})=
 \set{K}_1\oplus\cdots\oplus\set{K}_q,\ \ 
 \set{K}_j\subseteq\Sym^2\set{P}_j,\ \ 
 \set{P}_i\cap\set{P}_j=\{\vect{0}\}\ (i\ne j).
\]
Then there is a direct sum $\set{V}=\set{A}_1\oplus\cdots\oplus\set{A}_q$ such that
\[
 \set{A}_j\subseteq\set{P}_j,\ \ 
 \tens{T}=\tens{T}_1\oplus\cdots\oplus\tens{T}_q,\ \ 
 \tens{T}_j\in\Sym^3\set{A}_j.
\]
\end{lemma}

\begin{proof}
% LEAN: actual_polar_space_block_reconstruction constructs the dual
% decomposition, the actual cubic blocks, and their internal essential sum.
Write $\matr{F}=\operatorname{pol}_{\tens{T}}$ and
$\set{D}_j=\matr{F}^{-1}(\set{K}_j)$.
Conciseness makes $\matr{F}$ injective; since
$\operatorname{im}\matr{F}=\bigoplus_j\set{K}_j$, it follows that
$\set{V}^*=\bigoplus_j\set{D}_j$.
Let $\matr{\Pi}_j$ be the corresponding dual projections and put
$\matr{P}_j=\matr{\Pi}_j^*:\set{V}\to\set{V}$, using $\set{V}^{**}=\set{V}$.
For $\alpha\in\set{D}_i$, $\beta\in\set{D}_j$, and $i\ne j$, symmetry gives
\[
 \beta\mathbin{\lrcorner}\matr{F}(\alpha)
 =\alpha\mathbin{\lrcorner}\matr{F}(\beta)
 \in\set{P}_i\cap\set{P}_j=\{\vect{0}\}.
\]
Expanding arbitrary $\alpha,\beta$ along $\bigoplus_j\set{D}_j$, we obtain
$\beta\mathbin{\lrcorner}\matr{F}(\matr{\Pi}_j\alpha)
 =(\matr{\Pi}_j\beta)\mathbin{\lrcorner}\matr{F}(\alpha).$
Thus applying $\matr{P}_j$ in the first or the second tensor factor gives
the same tensor; symmetry allows it to be moved to the third factor as well.
It follows that
\[
 \tens{T}_j=(\matr{P}_j\otimes\operatorname{id}\otimes\operatorname{id})\tens{T}
 \in\Sym^3\set{V},\ \
 \operatorname{pol}_{\tens{T}_j}=\matr{F}\matr{\Pi}_j,\ \
 \tens{T}=\sum_j\tens{T}_j.
\]
Since $\matr{F}$ is injective, $\ker(\matr{F}\matr{\Pi}_j)=\ker\matr{\Pi}_j$.
By \cref{lem:contraction-essential-support}, we have
\[
 \set{A}_j:=\Ess(\tens{T}_j)
 =(\ker\matr{\Pi}_j)^\perp=\operatorname{im}\matr{P}_j.
\]
The identities $\sum_j\matr{P}_j=\operatorname{id}$ and
$\matr{P}_i\matr{P}_j=\delta_{ij}\matr{P}_j$ give $\set{V}=\bigoplus_j\set{A}_j$.
Moreover, $\operatorname{im}(\matr{F}\matr{\Pi}_j)
=\matr{F}(\set{D}_j)=\set{K}_j$, so every polar of $\tens{T}_j$ belongs to
$\Sym^2\set{P}_j$.  If $\alpha\in\set{P}_j^\perp$, then for every $\beta$,
\[
 \beta\mathbin{\lrcorner}
   (\alpha\mathbin{\lrcorner}\tens{T}_j)
 =\alpha\mathbin{\lrcorner}
   (\beta\mathbin{\lrcorner}\tens{T}_j)=\vect{0}.
\]
Hence $\alpha\mathbin{\lrcorner}\tens{T}_j=\tens{0}$, so
$\set{A}_j\subseteq\set{P}_j$ and
$\tens{T}_j\in\Sym^3\set{A}_j$.
\end{proof}

\begin{proposition}[Saturated excess-three descent]
\label{lem:excess-three-descent}
Let $\tens{T}\in\Sym^3\set{V}$ be concise, let $d=\dim\set{V}\ge5$, and
suppose $r(\tens{T})=d+3$.  If a minimum CP decomposition
\[
 \tens{T}=\sum_{i=1}^{d+3}\vect{a}_i\otimes\vect{b}_i\otimes\vect{c}_i
\]
satisfies $\dim\Span\{\vect{b}_i\wedge\vect{c}_i\}=3$, then either $s(\tens{T})=r(\tens{T})$ or
there are $\alpha\in\set{V}^*$, $\vect{\ell}\in\set{V}$, and $q\ne0$ such that
$\alpha\mathbin{\lrcorner}\tens{T}=q\vect{\ell}^2$ and 
$\alpha(\vect{\ell})\ne0.$ 
In the latter case
\[
 \tens{T}=\frac{q}{\alpha(\vect{\ell})}\vect{\ell}^3+\tens{R},\ \ 
 \tens{R}\in\Sym^3(\ker\alpha),\ \
 r(\tens{R})=r(\tens{T})-1,\ \ 
 d(\tens{R})=d(\tens{T})-1.
\]
\end{proposition}

\begin{proof}
Set $\matr{M}_i=\vect{b}_i\otimes\vect{c}_i$ and
$\vect{\omega}_i=\vect{b}_i\wedge\vect{c}_i$, and define
\[
 \matr{\Omega}:\RR^{r}\to\bigwedge^2\set{V},\quad
 \matr{\Omega}(\vect{t})=\sum_it_i\vect{\omega}_i,\ \ 
 \matr{\Phi}:\RR^r\to\set{V}\otimes\set{V},\quad
 \matr{\Phi}(\vect{t})=\sum_it_i\matr{M}_i.
\]
% LEAN: saturated_actual_polar_parameterization, using antisymmetric tensors
% b\otimes c-c\otimes b to represent the wedge columns.
If $\matr{A}:\RR^r\to\set{V}$ has columns $\vect{a}_i$, then, using the
standard-coordinate identification $(\RR^r)^*\cong\RR^r$, one has
$\operatorname{pol}_{\tens{T}}=\matr{\Phi}\matr{A}^*$, while symmetry gives
$\operatorname{im}\matr{A}^*\subseteq\ker\matr{\Omega}$.  Hence
\[
 d=\rank\operatorname{pol}_{\tens{T}}
 \le\rank\matr{A}^*\le\dim\ker\matr{\Omega}=r-3=d.
\]
Thus $\matr{A}^*:\set{V}^*\to\ker\matr{\Omega}$ is an isomorphism and
$\matr{\Phi}(\ker\matr{\Omega})=\set{P}_{\mathrm{pol}}(\tens{T})$.
Both spaces have dimension $d$, so the restriction of $\matr{\Phi}$ is an isomorphism.

Let $\set{Z}=\{i:\vect{\omega}_i=\vect{0}\}$ and $\set{N}=\{1,\ldots,r\}\setminus\set{Z}$.
Minimality makes every factor in the decomposition nonzero.  Thus, for
$j\in\set{Z}$, write $\matr{M}_j=q_j\vect{\ell}_j^2$ with
$q_j\ne0$ and $\vect{\ell}_j\ne\vect{0}$.  Since
$\vect{e}_j\in\ker\matr{\Omega}$, there is a unique $\alpha_j$ with
$\alpha_j\mathbin{\lrcorner}\tens{T}=q_j\vect{\ell}_j^2$.
Indeed,
$\matr{\Phi}(\matr{A}^*\alpha_j)=\matr{\Phi}(\vect{e}_j)$, so injectivity
gives $\matr{A}^*\alpha_j=\vect{e}_j$.  Hence the $\alpha_j$ are independent;
and if $\vect{\ell}_j$ and $\vect{\ell}_k$ were proportional, then
$\matr{\Phi}(\vect{e}_j)$ and $\matr{\Phi}(\vect{e}_k)$ would be proportional,
contradicting the injectivity of $\matr{\Phi}|_{\ker\matr{\Omega}}$.
Thus the $\vect{\ell}_j$ are pairwise nonproportional.  If some
$\alpha_j(\vect{\ell}_j)\ne0$,
\cref{lem:transverse-rank-one-deletion} gives the asserted cube deletion,
CP-rank drop, and dimension drop.

% LEAN: actual_mixed_contractions_commute and
% actual_distinct_rank_one_polars_annihilate.
Assume now that every $\alpha_j(\vect{\ell}_j)=0$.  Commutativity of mixed
contractions gives
$\alpha_k\mathbin{\lrcorner}(q_j\vect{\ell}_j^2)
=\alpha_j\mathbin{\lrcorner}(q_k\vect{\ell}_k^2)$, that is,
$q_j\alpha_k(\vect{\ell}_j)\vect{\ell}_j
=q_k\alpha_j(\vect{\ell}_k)\vect{\ell}_k$,
and therefore $\alpha_k(\vect{\ell}_j)=0$ for all $j,k\in\set{Z}$.
Let $\set{W}=\sum_{i\in\set{N}}\operatorname{supp}(\vect{\omega}_i)$.
The decomposition
$\ker\matr{\Omega}=\RR^{\set{Z}}\oplus
\ker(\matr{\Omega}|_{\RR^{\set{N}}})$
induces, with
$\set{K}_{\set{N}}=\matr{\Phi}
(\ker(\matr{\Omega}|_{\RR^{\set{N}}}))$,
\[
 \set{P}_{\mathrm{pol}}(\tens{T})=
 \bigoplus_{j\in\set{Z}}\RR\vect{\ell}_j^2\oplus\set{K}_{\set{N}},\ \ 
 \set{K}_{\set{N}}\subseteq\Sym^2\set{W}.
\]
Since $\operatorname{pol}_{\tens{T}}$ is an isomorphism from $\set{V}^*$
onto its image, this decomposition pulls back to
$\set{V}^*=\bigoplus_{j\in\set{Z}}\RR\alpha_j\oplus\set{D}_{\set{N}}$,
where $\set{D}_{\set{N}}=
\operatorname{pol}_{\tens{T}}^{-1}(\set{K}_{\set{N}})$.
If $\vect{\ell}_j\notin\set{W}$ and $\beta\in\set{D}_{\set{N}}$, then
$\alpha_j\mathbin{\lrcorner}(\beta\mathbin{\lrcorner}\tens{T})\in\set{W}$,
whereas commutativity of contractions gives
$\alpha_j\mathbin{\lrcorner}(\beta\mathbin{\lrcorner}\tens{T})
=q_j\beta(\vect{\ell}_j)\vect{\ell}_j$.  Hence
$\beta(\vect{\ell}_j)=0$.  We also have
$\alpha_k(\vect{\ell}_j)=0$ for every $k\in\set{Z}$.  The preceding
direct sum would therefore make every element of $\set{V}^*$ annihilate
$\vect{\ell}_j$, a contradiction.  Hence $\vect{\ell}_j\in\set{W}$ for
every $j\in\set{Z}$.  All second- and third-mode factors now belong to
$\set{W}$ and, by \cref{lem:minimum-cpd}, each of these two factor families
spans $\set{V}$.  Consequently, $\set{W}=\set{V}$.

Choose three observed bivectors that form a basis of
$\Span\{\vect{\omega}_i\}$, and let $\set{H}$ be the sum of their support
planes.  Every observed bivector belongs to $\bigwedge^2\set{H}$, so its
support lies in $\set{H}$.  Thus $\set{V}=\set{W}\subseteq\set{H}$ and
$d\le\dim\set{H}\le6$.

% LEAN: keepCoordinates_preserves_alternating_kernel and
% saturated_polar_blocks_from_cp_partition construct the supported polar blocks.
We construct the supported blocks used below directly from coordinates.
Suppose $\set{N}=\bigsqcup_\nu\set{I}_\nu$, with
$\vect{b}_i,\vect{c}_i\in\set{P}_\nu$ for $i\in\set{I}_\nu$, and suppose the spaces
$\set{L}_\nu=\Span\{\vect{\omega}_i:i\in\set{I}_\nu\}$ form a direct sum.
Let $\matr{R}_\nu$ retain precisely the coordinates in $\set{I}_\nu$ and put
$\set{K}=\ker(\matr{\Omega}|_{\RR^{\set{N}}})$.
For $\vect{t}\in\set{K}$, independence of the $\set{L}_\nu$ gives
\[
 \vect{0}=\sum_\nu\matr{\Omega}\matr{R}_\nu\vect{t},\ \
 \matr{\Omega}\matr{R}_\nu\vect{t}=\vect{0},\ \
 \sum_\nu\matr{R}_\nu|_{\set{K}}=\operatorname{id}_{\set{K}},\ \
 \matr{R}_\nu\matr{R}_\mu=\delta_{\nu\mu}\matr{R}_\nu.
\]
Since $\matr{\Phi}|_{\ker\matr{\Omega}}$ is injective, we obtain
\[
 \set{K}_{\set{N}}=\bigoplus_\nu\set{K}_\nu,
 \ \  \set{K}_\nu=\matr{\Phi}(\matr{R}_\nu\set{K}).
\]
The condition $\matr{\Omega}\matr{R}_\nu\vect{t}=\vect{0}$ makes the corresponding
quadratic tensor symmetric, so
\[
 \matr{\Phi}(\matr{R}_\nu\vect{t})
 =\frac12\sum_{i\in\set{I}_\nu}t_i
  (\vect{b}_i\otimes\vect{c}_i+\vect{c}_i\otimes\vect{b}_i)
 \in\Sym^2\set{P}_\nu.
\]
Here each parenthesis equals
$(\vect{b}_i+\vect{c}_i)^2-\vect{b}_i^2-\vect{c}_i^2$.
Set $\set{D}_\nu=\operatorname{pol}_{\tens{T}}^{-1}(\set{K}_\nu)$.
The polar-space decomposition gives
$\set{D}_{\set{N}}=\bigoplus_\nu\set{D}_\nu$ and hence
$\set{V}^*=\bigoplus_{j\in\set{Z}}\RR\alpha_j
\oplus\bigoplus_\nu\set{D}_\nu$.
% LEAN: actual_rank_one_polar_excludes_support and
% actual_rank_one_polar_support_allocation apply to these constructed blocks.
For $\beta\in\set{D}_\nu$, commutativity of mixed contractions gives
\[
q_j\beta(\vect{\ell}_j)\vect{\ell}_j
=\alpha_j\mathbin{\lrcorner}
  \operatorname{pol}_{\tens{T}}(\beta)\in\set{P}_\nu.
\]
Thus $\vect{\ell}_j\notin\set{P}_\nu$ implies
$\beta(\vect{\ell}_j)=0$ for all $\beta\in\set{D}_\nu$.
If $\vect{\ell}_j$ belonged to none of the $\set{P}_\nu$, these
vanishings and those from the zero-wedge polars would make the entire
space $\set{V}^*$ annihilate $\vect{\ell}_j\ne\vect{0}$, a contradiction.
Hence each $\vect{\ell}_j$ belongs to at least one support block;
when the blocks are pairwise disjoint, that block is unique.

Suppose $d=6$.  Choose observed bivectors
$\vect{\omega}_1,\vect{\omega}_2,\vect{\omega}_3$ forming a basis of
$\Span\{\vect{\omega}_i\}$, and put
$\set{P}_j=\operatorname{supp}(\vect{\omega}_j)$.  Their support planes
span $\set{W}=\set{V}$; since $d=6$, they form a direct sum of three planes.
% LEAN: simple_direction_in_independent_planes, via contraction-rank additivity.
For $\vect{\omega}=\sum_{j=1}^3\lambda_j\vect{\omega}_j$, put
$\matr{F}_{\vect{\omega}}(\beta)=\beta\mathbin{\lrcorner}\vect{\omega}$.
Independence of the support planes gives
\[
 \rank \matr{F}_{\vect{\omega}}=2\,|\{j:\lambda_j\ne0\}|.
\]
A nonzero decomposable $\vect{\omega}$ has contraction rank two;
hence exactly one coefficient is nonzero.  Direction grouping therefore gives
$\set{K}_{\set{N}}=\set{K}_1\oplus\set{K}_2\oplus\set{K}_3$ with
$\set{K}_i\subseteq\Sym^2\set{P}_i$.  Every $\vect{\ell}_j$ lies in exactly one \(\set{P}_i\),
because the planes are disjoint.  Adjoining each line
\(\RR\vect{\ell}_j^2\) to its unique block gives
\[
 \set{P}_{\mathrm{pol}}(\tens{T})=
 \widetilde{\set{K}}_1\oplus\widetilde{\set{K}}_2\oplus
 \widetilde{\set{K}}_3,\ \ 
 \widetilde{\set{K}}_i\subseteq\Sym^2\set{P}_i.
\]
% LEAN: saturated_blocks_from_cp_partition constructs the cubic blocks once
% the CP partition and its factor-support conditions have been established.
By \cref{lem:small-polar-blocks}, $\tens{T}$ is a direct sum of three
binary cubics, so $r(\tens{T})\le s(\tens{T})\le9=r(\tens{T})$.

It remains to treat $d=5$.  Let
$\set{S}=\{\RR\vect{\omega}_i:\vect{\omega}_i\ne\vect{0}\}$ be the set of
one-dimensional subspaces represented by the nonzero observed bivectors.
It spans $\set{L}_{bc}$ by definition.  If four members had representatives
such that every three were independent,
they would span the three-dimensional space $\set{L}_{bc}$;
\cref{lem:four-bivector-configuration}(i) would then put the support
of every observed bivector in a common subspace of dimension at most four, contrary to
\(\set{W}=\set{V}\) and \(d=5\).  Hence no such four members exist, and
\cref{lem:four-bivector-configuration}(ii) shows that \(\set{S}\) either has
exactly three members or has all but one lying in a two-dimensional subspace
$\set{M}\subseteq\set{L}_{bc}$.

Assume first that $|\set{S}|>3$.  Then the second alternative places at least
three members of $\set{S}$ in $\set{M}$.  Since $\set{S}$ spans the
three-dimensional space $\set{L}_{bc}$, exactly one member lies outside
$\set{M}$.  Choose independent decomposable representatives
$\vect{p},\vect{q}\in\set{M}$ and a third representative
$\vect{p}+\lambda\vect{q}$ with $\lambda\ne0$.  Its decomposability gives
$\vect{0}=(\vect{p}+\lambda\vect{q})^{\wedge2}=2\lambda\vect{p}\wedge\vect{q}$,
so the two support planes meet.  Put
$\set{P}=\operatorname{supp}(\vect{p})+
\operatorname{supp}(\vect{q})$.  Then $\dim\set{P}=3$ and
$\set{M}\subseteq\bigwedge^2\set{P}$.  The remaining direction has a
two-dimensional support $\set{Q}$.  Since $\set{P}+\set{Q}=\set{W}=\set{V}$
and $d=5$, one has $\set{V}=\set{P}\oplus\set{Q}$.  Direction grouping,
the preceding assertion about the $\vect{\ell}_j$, and the same allocation of
each \(\RR\vect{\ell}_j^2\) to its unique support block give
\(\set{P}_{\mathrm{pol}}(\tens{T})=
\widetilde{\set{K}}_{\set{P}}\oplus
\widetilde{\set{K}}_{\set{Q}}\), with
\(\widetilde{\set{K}}_{\set{P}}\subseteq\Sym^2\set{P}\) and
\(\widetilde{\set{K}}_{\set{Q}}\subseteq\Sym^2\set{Q}\).  Thus
\cref{lem:small-polar-blocks} gives a ternary plus a binary cubic.  Hence
$r(\tens{T})\le s(\tens{T})\le5+3=8=r(\tens{T})$.

If $\set{S}$ consists of exactly three members and two support planes meet,
their sum is three-dimensional and, since the three supports span
$\set{W}=\set{V}$, is complementary to the third support plane.  The same
two-block construction gives
$r(\tens{T})\le s(\tens{T})\le5+3=8=r(\tens{T})$.
If the three support planes are
pairwise disjoint, direction grouping and unique allocation of the
zero-wedge squares split the entire polar space into three supported
blocks.  Lemma~\ref{lem:small-polar-blocks} gives
$\set{V}=\set{A}_1\oplus\set{A}_2\oplus\set{A}_3$ with
$\dim\set{A}_i\le2$.  Since the dimensions sum to $\dim\set{V}=5$, they are
$2,2,1$ up to order, so $s(\tens{T})\le3+3+1=7$, contradicting $r(\tens{T})=8$.
Thus every remaining branch has $s(\tens{T})=r(\tens{T})$.
\end{proof}

\subsection{Proof of the zero-gap proposition}

\begin{proof}[Proof of \cref{prop:excess-three-zero}]
Assume otherwise, and choose a counterexample $\tens{T}$ with
$r=r(\tens{T})$ minimal.  Then $\tens{T}\ne\tens{0}$.  By
\cref{lem:concise-reduction}, take
$\set{V}=\Ess(\tens{T})$, and put $d=\dim\set{V}$ and $e=r-d\le3$.
Fix a minimum CP decomposition $\mathscr{D}$.  If $\mathscr{D}$ contains
a displayed cube, \cref{lem:displayed-cube-deletion} gives a residual
$\tens{R}$ with $r(\tens{R})<r$, $e(\tens{R})\le e$, and
$\DeltaT(\tens{T})\le\DeltaT(\tens{R})=0$ by minimality (with the zero
residual being immediate), a contradiction.  Thus $\mathscr{D}$ is cube-free.

For any minimum decomposition $\mathscr{E}$ of $\tens{T}$ and any mode
pair $xy$, let $\rho_{xy}(\mathscr{E})$ denote its wedge rank.  The
matrix-kernel argument in the proof of
\cref{lem:wedge-support-bounds}, together with
\cref{lem:minimum-cpd}, gives $\rho_{xy}(\mathscr{E})\le r-d=e\le3$. 
If all three wedge ranks of $\mathscr{D}$ are at most two, then, after
relabeling the modes, \cref{lem:low-wedge-partialization} gives a minimum
decomposition $\mathscr{E}$ of the form
$\tens{T}=\sum_{i=1}^{r}\vect{a}_i\otimes\vect{\ell}_i^{\otimes2}$.
If $\mathscr{E}$ contained a displayed cube, the preceding deletion argument
would give the same contradiction.  Hence $\mathscr{E}$ is cube-free and
$\vect{a}_i\wedge\vect{\ell}_i\ne\vect{0}$ for every $i$.  Set
$\sigma=\dim\Span\{\vect{a}_i\wedge\vect{\ell}_i\}$; the same kernel
argument gives $\sigma\le e$.  If $\sigma\le2$,
\cref{lem:partial-low-second-wedge} gives $s(\tens{T})=r$, a contradiction.
If $\sigma=3$, interchanging the first two factors gives the minimum
decomposition $\sum_i\vect{\ell}_i\otimes\vect{a}_i\otimes\vect{\ell}_i$,
whose $bc$-wedge rank is three.  Consequently, every remaining case has a minimum
cube-free decomposition with one mode-pair wedge rank equal to three.

It follows that $e=3$ and $r=d+3$.  For $1\le d\le3$, the inequalities
$r\le s(\tens{T})\le\Mthree(d)$ contradict
$\Mthree(1)=1$, $\Mthree(2)=3$, and $\Mthree(3)=5$; for $d=4$,
\cref{lem:four-variable-seven} gives $r=7\le s(\tens{T})\le7$, again a
contradiction.
Finally, if $d\ge5$, relabel the modes and apply
\cref{lem:excess-three-descent}.  Its equality branch contradicts the
choice of $\tens{T}$.  Otherwise
$\tens{T}=\lambda\vect{\ell}^{3}+\tens{R}$ with
$r(\tens{R})=r-1$ and $d(\tens{R})=d-1$, so $e(\tens{R})=3$ and
minimality gives $\DeltaT(\tens{R})=0$.  Since
$s(\tens{T})\le1+s(\tens{R})$, it follows that
$\DeltaT(\tens{T})\le\DeltaT(\tens{R})=0$, again a contradiction.
\end{proof}

\begin{proof}[Proof of \cref{thm:rank-excess-main}]
If $e(\tens{T})\leq3$, apply \cref{prop:excess-three-zero}; if
$e(\tens{T})\geq4$, apply \cref{prop:linear-excess-main}.
These are the two branches of \eqref{eq:def-C-excess}.
\end{proof}

\section{Bounds controlled by essential dimension}
\label{sec:proof-dimension}

This section proves the two propositions supporting
\cref{thm:essential-dimension-main}.  We first derive the hyperplane-reduction
estimate from polar completion, then establish the shifted maximum-rank
estimate and the exceptional six-variable endpoint.  The final subsection
assembles these inputs into the two dimension bounds.
For $\tens{T}\in\Sym^3(\set{U})$ and $\alpha\in\set{U}^*$, write
$\alpha\mathbin{\lrcorner}\tens{T}
\eqdef(\alpha\otimes\operatorname{id}\otimes\operatorname{id})(\tens{T})$.
Under the polynomial identification, this contraction is
$\frac13\partial_\alpha\tens{T}$; we use the same notation for a
symmetric tensor and its associated homogeneous polynomial.

\subsection{Hyperplane reduction by transverse polars}

As before, $\vect{v}^2$ denotes $\vect{v}\otimes\vect{v}$. 
The polar-completion lemma constructs a cubic with a prescribed
quadratic polar.  A transverse diagonalization costs exactly $h$
cubes; adding and subtracting one auxiliary cube handles the remaining
case.  This is the rank-sensitive form of the direction-avoidance step
related to \cite{BialynickiBirulaSchinzel2008,BialynickiBirulaSchinzel2011corrigendum,jelisiejew2014upper}.

\begin{lemma}[Polar completion]
\label{lem:polar-completion}
Let $\set{U}$ be a real vector space of dimension $m+1$, where
$m\geq0$, let $\tens{T}\in\Sym^3(\set{U})$, and fix
$\alpha\in\set{U}^*$ such that
$\tens{Q}\eqdef\alpha\mathbin{\lrcorner}\tens{T}\neq\tens{0}$.
Set $\set{H}\eqdef\ker\alpha$ and $h\eqdef\rank(\tens{Q})$.  Then
\begin{equation}
s(\tens{T})\leq h+2+\Mthree(m).
\label{eq:polar-completion}
\end{equation}
If $\Ess(\tens{Q})\not\subseteq\set{H}$, then
\begin{equation}
s(\tens{T})\leq h+\Mthree(m).
\label{eq:polar-transverse}
\end{equation}
\end{lemma}

\begin{proof}
Put $\set{H}=\ker\alpha$ and $\set{W}=\Ess(\tens{Q})$.
Choose $\vect{z}\in\set{U}$ with $\alpha(\vect{z})=1$.
Under $\set{U}=\set{H}\oplus\RR\vect{z}$, contraction by
$\alpha$ is one third of differentiation with respect to the coordinate $\vect{z}$.
Its kernel on cubics is therefore precisely $\Sym^3(\set{H})$. 
First suppose $\set{W}\not\subseteq\set{H}$.
Sylvester's law of inertia gives
$\tens{Q}=\sum_{i=1}^h\varepsilon_i\vect{v}_i^2$, where
$\varepsilon_i\in\{1,-1\}$ and the $\vect{v}_i$ form a basis of
$\set{W}$.  Since $\alpha|_{\set{W}}\ne0$, at least one basis vector
has nonzero $\alpha$-value.  If $\alpha(\vect{v}_j)=0$, choose $i$
with $\alpha(\vect{v}_i)\ne0$ and change only this pair.  When
$\varepsilon_i=\varepsilon_j$, replace the pair by
$(\vect{v}_i+\vect{v}_j)/\sqrt2$ and
$(\vect{v}_i-\vect{v}_j)/\sqrt2$.  When
$\varepsilon_i=-\varepsilon_j$, use instead
$\sqrt2\vect{v}_i+\vect{v}_j$ and
$\vect{v}_i+\sqrt2\vect{v}_j$.  These changes preserve the signed
quadratic sum and replace one zero $\alpha$-value by nonzero values
without introducing a new zero.  Repeating this operation gives
$\tens{Q}=\sum_{i=1}^h c_i\vect{v}_i^2$ with $c_i\ne0$ and
$\alpha(\vect{v}_i)\ne0$ for every $i$.  Put
$\tens{S}=\sum_{i=1}^h
(c_i/\alpha(\vect{v}_i))\vect{v}_i^3$.
Then $\alpha\mathbin{\lrcorner}\tens{S}=\tens{Q}$, so
$\tens{T}-\tens{S}\in\Sym^3(\set{H})$.
Every coefficient in $\tens{S}$ can be absorbed by a real cube root;
consequently $s(\tens{T})\le h+\Mthree(m)$.
This proves \eqref{eq:polar-transverse}.

Now suppose $\set{W}\subseteq\set{H}$ and set
$\widetilde{\tens{T}}=\tens{T}+\vect{z}^3$.
Its polar is $\widetilde{\tens{Q}}=\tens{Q}+\vect{z}^2$.
Relative to $\set{H}\oplus\RR\vect{z}$, this quadratic tensor is
block diagonal, so it has rank $h+1$ and essential space
$\set{W}\oplus\RR\vect{z}\not\subseteq\set{H}$.
The transverse case applied to $\widetilde{\tens{T}}$ gives
$s(\widetilde{\tens{T}})\le h+1+\Mthree(m)$.
Since $-\vect{z}^3=(-\vect{z})^3$, subtracting this auxiliary cube
yields $s(\tens{T})\le h+2+\Mthree(m)$.
The two cases prove \eqref{eq:polar-completion}.
\end{proof}

The next estimate records the two bounds obtained from the same
minimum CP decomposition.  Its two functionals need not coincide, so
the minimum retains the better choice.

\begin{lemma}[Adaptive hyperplane estimate]
\label{lem:adaptive-bound}
Let $\set{U}$ be a real vector space, let
$\tens{T}\in\Sym^3(\set{U})$ be nonzero and concise, and let
$\tens{T}=\sum_{j=1}^{r(\tens{T})}
\vect{a}_j\otimes\vect{b}_j\otimes\vect{c}_j$ be any minimum CP
decomposition.  For the first factor matrix
$\matr{A}=(\vect{a}_1\ \cdots\ \vect{a}_{r(\tens{T})})$, define
$\eta_{\matr{A}}$ to be the maximum of
$\#\{j:\alpha(\vect{a}_j)=0\}$ over nonzero
$\alpha\in\set{U}^*$.  Define $\eta_{\matr{A}}^{\pitchfork}$ by the
same maximum restricted to functionals satisfying
$\alpha^2\mathbin{\lrcorner}\tens{T}\ne\vect{0}$, where the second
contraction is successive; this restricted set is nonempty by
polarization.  Then
\begin{equation}
\DeltaT(\tens{T})
\leq\Mthree\bigl(d(\tens{T})-1\bigr)
+\min\{2-\eta_{\matr{A}},-\eta_{\matr{A}}^{\pitchfork}\}.
\label{eq:adaptive-bound}
\end{equation}
\end{lemma}

\begin{proof}
Choose a nonzero functional $\alpha$ attaining $\eta_{\matr{A}}$.
Set $\tens{Q}=\alpha\mathbin{\lrcorner}\tens{T}$ and
$h=\rank(\tens{Q})$.
The first flattening of a concise tensor is injective, so
$\tens{Q}\ne\tens{0}$.
Contracting the fixed minimum CP decomposition leaves
$r(\tens{T})-\eta_{\matr{A}}$ nonzero rank-one matrix summands.
Matrix-rank subadditivity thus gives
$1\le h\le r(\tens{T})-\eta_{\matr{A}}$; equality is not needed.
Applying \cref{lem:polar-completion} with $m=d(\tens{T})-1$ yields
$\DeltaT(\tens{T})\le
\Mthree(d(\tens{T})-1)+2-\eta_{\matr{A}}$.

Choose a new functional $\alpha$ attaining
$\eta_{\matr{A}}^{\pitchfork}$, and reset
$\tens{Q}=\alpha\mathbin{\lrcorner}\tens{T}$ and
$h=\rank(\tens{Q})$.  Matrix-rank subadditivity gives
$h\le r(\tens{T})-\eta_{\matr{A}}^{\pitchfork}$.  The condition
$\alpha^2\mathbin{\lrcorner}\tens{T}\ne\vect{0}$ implies
$\Ess(\tens{Q})\not\subseteq\ker\alpha$, so the transverse part of
\cref{lem:polar-completion} gives
$\DeltaT(\tens{T})\le
\Mthree(d(\tens{T})-1)-\eta_{\matr{A}}^{\pitchfork}$.
Taking the better of these two estimates proves
\eqref{eq:adaptive-bound}.
\end{proof}

The next lemma supplies the hyperplane required by the transverse
branch of polar completion.  Its proof selects a basis among the first
factor columns and uses polarization when all pure second contractions
vanish.

\begin{lemma}[Transverse column hyperplane]
\label{lem:transverse-column-hyperplane}
Let $\tens{T}\in\Sym^3(\set{U})$ be
nonzero and concise, and suppose that
$r(\tens{T})>d=d(\tens{T})$.  For every minimum CP
decomposition with first factor matrix $\matr{A}$, one has
$\eta_{\matr{A}}^{\pitchfork}\geq d-1$.
\end{lemma}

\begin{proof}
Choose independent columns
$\vect{a}_1,\ldots,\vect{a}_d$ of $\matr{A}$ and one further column
$\vect{a}_{d+1}=\sum_{i=1}^d c_i\vect{a}_i$.  Let
$\epsilon_1,\ldots,\epsilon_d$ be the dual basis, and set
$\vect{b}(\beta,\gamma)
=(\beta\otimes\gamma\otimes\operatorname{id})(\tens{T})$ and
$\vect{q}(\beta)=\vect{b}(\beta,\beta)$.  If
$\vect{q}(\epsilon_i)\ne\vect{0}$ for some $i$, then $\epsilon_i$ is
admissible for $\eta_{\matr{A}}^{\pitchfork}$ and annihilates the
other $d-1$ chosen basis columns.

Suppose that $\vect{q}(\epsilon_i)=\vect{0}$ for every $i$.  Since
$\tens{T}$ is symmetric, these vanishings kill every coefficient with
a repeated index.  Since $\tens{T}\ne\tens{0}$, there are therefore distinct
$i,j,k$ with
$\tens{T}(\epsilon_i,\epsilon_j,\epsilon_k)\ne0$.  In particular,
$\vect{b}(\epsilon_p,\epsilon_q)\ne\vect{0}$ for every pair
$p,q\in\{i,j,k\}$, $p\ne q$.  Among $c_i,c_j,c_k$, choose $c_p,c_q$
that are either both zero or both nonzero.  In the first case take
$\alpha=\epsilon_p+\epsilon_q$; in the second take
$\alpha=c_q\epsilon_p-c_p\epsilon_q$.  Then $\alpha$ annihilates
$\vect{a}_{d+1}$ and the $d-2$ basis columns other than
$\vect{a}_p,\vect{a}_q$, while
$\vect{q}(\alpha)$ is a nonzero scalar multiple of
$\vect{b}(\epsilon_p,\epsilon_q)$.  Thus $\alpha$ is admissible for
$\eta_{\matr{A}}^{\pitchfork}$ and annihilates at least $d-1$
selected columns.
\end{proof}

The preceding two lemmas combine into the hyperplane-reduction step
used in the final assembly.

\begin{proposition}[Hyperplane reduction for the rank gap]
\label{prop:hyperplane-gap-reduction}
Let $\set{U}$ be a real vector space and let
$\tens{0}\ne\tens{T}\in\Sym^3(\set{U})$.  Put
$d=d(\tens{T})$.  Then
$\DeltaT(\tens{T})\le\Mthree(d-1)-d+1$.
\end{proposition}

\begin{proof}
By \cref{lem:concise-reduction}, we may replace $\set{U}$ by
$\Ess(\tens{T})$ and assume that $\tens{T}$ is concise.  The first
flattening gives $r(\tens{T})\ge d$.  If $r(\tens{T})=d$, then
\cref{prop:excess-three-zero}, applied with $e(\tens{T})=0$, gives
$\DeltaT(\tens{T})=0$.  The asserted right-hand side is nonnegative:
$\Mthree(d-1)\ge d-1$ for $d\ge2$ by the Fermat cubic, while it is zero
for $d=1$.  If $r(\tens{T})>d$, choose a minimum CP decomposition with
first factor matrix $\matr{A}$.  By
\cref{lem:transverse-column-hyperplane},
$\eta_{\matr{A}}^{\pitchfork}\ge d-1$; the transverse branch of
\eqref{eq:adaptive-bound} now gives the assertion.
\end{proof}

The following remark clarifies the order of the choices in the adaptive
bound, in particular the dependence on the selected minimum decomposition.
\begin{remark}[Dependence on the minimum decomposition]
For every chosen minimum CP decomposition there are suitable
functionals attaining $\eta_{\matr{A}}$ and
$\eta_{\matr{A}}^{\pitchfork}$, and both adaptive numbers depend on
that decomposition.  The proof does not assert that a single
hyperplane works simultaneously for all minimum decompositions.
\end{remark}

\subsection{Shifted maximum-rank estimate}

The following maximum-rank estimate closes the hyperplane-reduction
bound and supplies the main input for the dimension-only theorem.  The
numerical value of $\gthree$ is supplied by the Alexander--Hirschowitz theorem
\cite{hirschowitz1995polynomial,brambilla2008alexander}.

\begin{proposition}[Shifted maximum real Waring rank]
\label{prop:shifted-maximum-rank}
For every $d\geq1$, we have
$$\Mthree(d)\leq\Vthree(d)=2\gthree(d-1)+1.$$
\end{proposition}

\begin{remark}[Comparison with existing maximum-rank bounds]
The estimate above differs essentially from the same-dimension bound
$\Mthree(d)\leq2\gthree(d)$ of Blekherman and Teitler.  Their
Waring-rank specialization gives $\Mthree(5)\leq15$ and
$\Mthree(6)\leq20$
\cite[Corollary~9]{blekherman2015maximum}, whereas
\cref{prop:shifted-maximum-rank} gives $\Mthree(5)\leq11$ and
$\Mthree(6)\leq17$.  The Alexander--Hirschowitz theorem supplies
$\gthree(m)=\lceil(m+1)(m+2)/6\rceil$ for $m\geq1$, $m\ne5$, together
with the exceptional value $\gthree(5)=8$, but it determines generic
complex rank rather than maximum real rank
\cite{hirschowitz1995polynomial,brambilla2008alexander}.  Related
maximum-rank refinements of Bernardi and Staffolani and of Buczyński,
Han, Mella, and Teitler are established over algebraically closed
fields under additional geometric hypotheses on the variety or its
secant loci; they therefore do not imply the shifted real estimate
above \cite{bernardi2022maximal,buczynski2018locus}.  In dimension
three, the exact value $\Mthree(3)=5=2\gthree(2)+1$ was determined by
Banchi \cite[Theorem~2.1]{banchi2015ternary}, but this low-dimensional
identity does not yield the general assertion.
\end{remark}

Its proof uses two independent ingredients: regular completion for the real
cube-sum map and a rational deformation of the quadratic polar to transverse
directions.
For a real vector space $\set{H}$ and an integer $g\geq1$, write
$\Phi_g:\set{H}^g\to\Sym^3(\set{H})$ for the cube-sum map
$\Phi_g(\vect{u}_1,\ldots,\vect{u}_g)=\sum_i\vect{u}_i^3$.
% LEAN: only polynomial nonvanishing is covered here; dominance and submersion remain open.
The Alexander--Hirschowitz theorem supplies the full-rank Jacobian used
below; $\operatorname{GL}(\set{H})$-equivariance then gives the required
regular completion.

\begin{lemma}[Regular cube completion]
\label{lem:regular-cube-completion}
Let $\dim\set{H}=m\geq1$ and $g=\gthree(m)$.  Then $g\geq m$.
Moreover, every ordered independent family
$\vect{p}_1,\ldots,\vect{p}_q\in\set{H}$ with $0\leq q\leq m$
extends to a point
$(\vect{p}_1,\ldots,\vect{p}_q,
\vect{p}_{q+1},\ldots,\vect{p}_g)$ at which
$D\Phi_g$ is surjective.
\end{lemma}

\begin{proof}
By the definition of $g=\gthree(m)$ and the Alexander--Hirschowitz theorem,
with coefficients absorbed into complex cube roots, the complexification of
$\Phi_g$ is dominant.  In characteristic zero,
generic smoothness, equivalently the Jacobian criterion for a dominant
polynomial map, therefore gives a complex $g$-tuple at which $D\Phi_g$ is
surjective.  Thus, for
$N=\dim\Sym^3(\set{H})=\binom{m+2}{3}$, one has $gm\ge N$, and hence
$$g\ge\frac{(m+1)(m+2)}{6}\ge m.$$  
Since $\Phi_g$ is defined over $\RR$, some
$N\times N$ Jacobian minor $J$ is a nonzero polynomial with real
coefficients.  Multiplying $J$ by a nonzero $q\times q$ coordinate
minor expressing the independence of the first $q$ vectors (or by
$1$ if $q=0$) still gives a nonzero real polynomial.  It is nonzero at
some real point, so there is a regular $g$-tuple whose first $q$
entries are independent.  The equivariance
$$\Phi_g(\matr{G}\vect{u}_1,\ldots,\matr{G}\vect{u}_g)
=\Sym^3(\matr{G})\Phi_g(\vect{u}_1,\ldots,\vect{u}_g)$$ and the transitivity of
$\operatorname{GL}(\set{H})$ on ordered $q$-frames give the asserted
completion of any prescribed family.
\end{proof}

The required deformation can be expressed rationally, avoiding a matrix
exponential and second-order remainder estimates.

\begin{lemma}[Rational orthogonal deformation]
\label{lem:cayley-deformation}
Let $\matr{B},\matr{A}$ be real square matrices with
$\matr{A}^{\mathsf T}\matr{B}+\matr{B}\matr{A}=\matr{0}$.
For all sufficiently small $\varepsilon$, put
\[
 \matr{K}_{\varepsilon}=\matr{I}-\frac{\varepsilon}{2}\matr{A},
 \ \ 
 \matr{G}_{\varepsilon}
 =\left(\matr{I}+\frac{\varepsilon}{2}\matr{A}\right)
       \matr{K}_{\varepsilon}^{-1}.
\]
Then $\matr{K}_{\varepsilon}$ is invertible and
\[
 \matr{G}_{\varepsilon}^{\mathsf T}\matr{B}\matr{G}_{\varepsilon}=\matr{B},
 \ \ 
 \matr{G}_{\varepsilon}-\matr{I}
 =\varepsilon\matr{A}\matr{K}_{\varepsilon}^{-1},
 \ \ 
 \matr{A}\matr{K}_{\varepsilon}^{-1}\longrightarrow\matr{A}
 \quad(\varepsilon\to0).
\]
\end{lemma}

\begin{proof}
% LEAN: cayley_preserves_bilinear, cayleyDenominator_eventually_unit,
% cayley_first_order_exact, cayley_first_order_limit; assembled in cayley_family_spec.
Continuity and $\det\matr{K}_0=1$ give invertibility near zero.
Writing $\matr{L}_{\varepsilon}=\matr{I}+\varepsilon\matr{A}/2$,
the skew relation gives
$\matr{L}_{\varepsilon}^{\mathsf T}\matr{B}\matr{L}_{\varepsilon}
=\matr{K}_{\varepsilon}^{\mathsf T}\matr{B}\matr{K}_{\varepsilon}$.
Multiplying by $\matr{K}_{\varepsilon}^{-\mathsf T}$ and
$\matr{K}_{\varepsilon}^{-1}$ proves preservation of $\matr{B}$.
The difference identity follows from
$\matr{L}_{\varepsilon}-\matr{K}_{\varepsilon}=\varepsilon\matr{A}$;
the limit follows from continuity of inversion at $\matr{I}$.
\end{proof}

\begin{proof}[Proof of \cref{prop:shifted-maximum-rank}]
The assertion is immediate for $d=1$.  Let $d\geq2$, put $m=d-1$
and $g=\gthree(m)$.  The zero form is harmless, so take a nonzero
$\tens{T}\in\Sym^3(\set{U})$, where $\dim\set{U}=d$.  Since its
associated cubic on $\set{U}^*$ is nonzero, choose $\alpha\in\set{U}^*$ such that
$\tens{T}(\alpha,\alpha,\alpha)\ne0$, and set
$\set{H}=\ker\alpha$ and
$\tens{Q}=\alpha\mathbin{\lrcorner}\tens{T}$.  Write
$h=\rank\tens{Q}$ and $\set{W}=\Ess(\tens{Q})$.  Regard
$\tens{Q}$ as the isomorphism
$\tens{Q}^{\sharp}:\set{W}^*\to\set{W}$, let $\tens{B}$ be its inverse
symmetric bilinear form on $\set{W}$, put
$\alpha_{\set{W}}=\alpha|_{\set{W}}$, and set
$\vect{n}=\tens{Q}^{\sharp}\alpha_{\set{W}}$.  Then
$\tens{B}(\vect{n},\vect{w})=\alpha_{\set{W}}(\vect{w})$ for
$\vect{w}\in\set{W}$.  Moreover,
\[
\tens{B}(\vect{n},\vect{n})
=\tens{Q}(\alpha_{\set{W}},\alpha_{\set{W}})
=\tens{T}(\alpha,\alpha,\alpha)\ne0.
\]
Consequently
$\set{K}=\set{W}\cap\set{H}=\vect{n}^{\perp_{\tens{B}}}$ is nondegenerate
and has dimension $q=h-1$.  Since $\set{K}\subseteq\set{H}$, one has
$q\le m$, as required in \cref{lem:regular-cube-completion}.

Choose a $\tens{B}$-orthogonal basis
$\vect{w}_1,\ldots,\vect{w}_q,\vect{w}_h=\vect{n}$ of $\set{W}$.
Set $\beta_i=\tens{B}(\vect{w}_i,\vect{w}_i)$,
$\delta_i=\beta_i^{-1}$, and
$\rho=\alpha(\vect{w}_h)\ne0$.  Thus
$\tens{Q}=\sum_{i=1}^h\delta_i\vect{w}_i^2$.  If $q>0$, define
an endomorphism $\matr{A}$ of $\set{W}$ by
\[
\matr{A}\vect{w}_i=\vect{w}_h\ (i\leq q),
\ \ 
\matr{A}\vect{w}_h=-\sum_{i=1}^q
\frac{\beta_h}{\beta_i}\vect{w}_i;
\]
if $q=0$, take $\matr{A}=\matr{0}$.  Directly on this basis, $\matr{A}$ is $\tens{B}$-skew.
Apply \cref{lem:cayley-deformation} in this basis and put
$\vect{v}_i(\varepsilon)=\matr{G}_{\varepsilon}\vect{w}_i$.
Preservation of $\tens{B}$ also preserves its inverse tensor, giving
$\tens{Q}=\sum_i\delta_i\vect{v}_i(\varepsilon)^2$.
The exact difference identity and $\matr{K}_{\varepsilon}^{-1}\to\matr{I}$ give
\[
 \frac{\alpha(\vect{v}_i(\varepsilon))}{\varepsilon}\longrightarrow\rho
 \quad(i\le q),\ \ 
 \alpha(\vect{v}_h(\varepsilon))\longrightarrow\rho\ne0.
\]
For sufficiently small nonzero $\varepsilon$, all these values are
nonzero.  Define
\[
\tens{S}_{\varepsilon}
=\sum_{i=1}^h
\frac{\delta_i}{\alpha(\vect{v}_i(\varepsilon))}
\vect{v}_i(\varepsilon)^3,
\ \ 
\tens{R}_{\varepsilon}=\tens{T}-\tens{S}_{\varepsilon}.
\]
Our contraction normalization gives
$\alpha\mathbin{\lrcorner}\vect{v}^3=\alpha(\vect{v})\vect{v}^2$.
Thus $\alpha\mathbin{\lrcorner}\tens{S}_{\varepsilon}=\tens{Q}$
and, since the kernel of contraction by $\alpha$ on cubics is
$\Sym^3(\set{H})$, one has
$\tens{R}_{\varepsilon}\in\Sym^3(\set{H})$.  Moreover,
\[
\begin{aligned}
\frac{\vect{v}_i(\varepsilon)-\vect{w}_i}{\varepsilon}
 &\longrightarrow\vect{w}_h,\ \ 
\frac{\varepsilon\delta_i}
 {\alpha(\vect{v}_i(\varepsilon))}\vect{v}_i(\varepsilon)^3
 \longrightarrow\frac{\delta_i}{\rho}\vect{w}_i^3
 (i\le q),\\
\frac{\varepsilon\delta_h}
 {\alpha(\vect{v}_h(\varepsilon))}\vect{v}_h(\varepsilon)^3
 &\longrightarrow\tens{0},\ \
\varepsilon\tens{R}_{\varepsilon}
 \longrightarrow-\tens{P},\ \ 
\tens{P}=\rho^{-1}\sum_{i=1}^q\delta_i\vect{w}_i^3.
\end{aligned}
\]
Absorb the nonzero real coefficients into cube roots and write
$\tens{P}=\sum_{i=1}^q\vect{p}_i^3$, where the
$\vect{p}_i$ are independent; for $q=0$, this is the empty sum.

By \cref{lem:regular-cube-completion}, we extend the independent family
$(-\vect{p}_i)_{i=1}^q$ to a regular $g$-tuple
$(\vect{\ell}_i)_{i=1}^g$.  Appending
$(-\vect{\ell}_i)_{i=q+1}^g$ gives a
$(2g-q)$-cube representation of $-\tens{P}$.  The differential
of the corresponding cube-sum map is surjective there because its
first $g$ entries already form a regular tuple.  The real submersion
theorem \cite{lee2012smooth} therefore gives a neighborhood of $-\tens{P}$ consisting of
cubics of real Waring rank at most $2g-q$.  For sufficiently small
finite $\varepsilon\ne0$, it contains
$\varepsilon\tens{R}_{\varepsilon}$.  Multiplication by a nonzero
real scalar does not change Waring rank, so
$s(\tens{R}_{\varepsilon})\leq2g-q$.

Finally, $\tens{S}_{\varepsilon}$ is a sum of at most $h$ real cubes.
Since $q=h-1$, we conclude that
$s(\tens{T})\leq h+2g-q=2g+1=\Vthree(d)$.
\end{proof}

\subsection{Exceptional six-variable endpoint}

The following six-dimensional lemma provides the geometric input for
the exceptional estimate $\Mthree(6)\le16$.  Its proof combines the
Stiefel--Whitney classes of real projective space
\cite[Section~4]{milnor1974classes}, the equality between semialgebraic
dimension and the dimension of the real Zariski closure
\cite[Proposition~2.8.2]{bochnak1998real}, and Bertini irreducibility
\cite[Section~37.32]{stacks-project}.

\begin{lemma}[Singular transverse polar in six variables]
\label{lem:singular-transverse-six}
Let $\dim\set{U}=6$ and
$\tens{0}\ne\tens{T}\in\Sym^3(\set{U})$.  There is
$\alpha\in\set{U}^*\setminus\{0\}$ such that, for
$\tens{Q}_{\alpha}=\alpha\mathbin{\lrcorner}\tens{T}$,
we have $\rank\tens{Q}_{\alpha}\le5$ and
$\alpha^2\mathbin{\lrcorner}\tens{T}\ne\vect{0}$.
Equivalently, $\det\tens{Q}_{\alpha}=0$ and
$\Ess(\tens{Q}_{\alpha})\not\subseteq\ker\alpha$.
\end{lemma}

\begin{proof}
On $\set{U}^*\cong\RR^6$, write
$f(\vect{x})=\tens{T}(\vect{x},\vect{x},\vect{x})$,
$\vect{q}(\vect{x})=\nabla f(\vect{x})$, and
$\matr{H}_{\vect{x}}=\operatorname{Hess}f(\vect{x})$.
Euler's identities give
$\langle\vect{q}(\vect{x}),\vect{x}\rangle=3f(\vect{x})$ and
$\matr{H}_{\vect{x}}\vect{x}=2\vect{q}(\vect{x})$.
If $\det\matr{H}_{\vect{x}}\equiv0$, choose $\vect{x}\ne\vect{0}$ with
$\vect{q}(\vect{x})\ne\vect{0}$.
If $f=\vect{\ell}\tens{Q}$ is reducible and concise, choose
$\vect{x}\in\ker\vect{\ell}$ with $\tens{Q}(\vect{x})\ne0$.
Such a point exists, since otherwise $\tens{Q}$ is divisible by
$\vect{\ell}$ and $f$ has at most two essential variables.  Then
\[
\vect{q}(\vect{x})=\tens{Q}(\vect{x})\vect{\ell}\ne\vect{0},\ \
\matr{H}_{\vect{x}}=
\vect{\ell}\otimes\nabla\tens{Q}(\vect{x})
+\nabla\tens{Q}(\vect{x})\otimes\vect{\ell},\ \
\rank\matr{H}_{\vect{x}}\le2.
\]
The nonconcise case also has $\det\matr{H}_{\vect{x}}\equiv0$.
A real irreducible cubic is absolutely irreducible: conjugation
permutes its complex irreducible factors, and odd degree forces a
fixed factor that descends to $\RR$.  It remains to consider a
concise, absolutely irreducible cubic with
$\det\matr{H}_{\vect{x}}\not\equiv0$.

Assume that
$\det\matr{H}_{\vect{x}}=0\Rightarrow\vect{q}(\vect{x})=\vect{0}$, and set
$$\set{Z}_{\RR}
=\{[\vect{x}]\in\mathbb{RP}^5:\vect{q}(\vect{x})=\vect{0}\}
=\{[\vect{x}]\in\mathbb{RP}^5:\det\matr{H}_{\vect{x}}=0\}.$$
On $\mathbb{RP}^5\setminus\set{Z}_{\RR}$, the map
$\Phi([\vect{x}])=\vect{q}(\vect{x})/\|\vect{q}(\vect{x})\|
\in\mathbb{S}^5$ is well-defined, since $\vect{q}$ is homogeneous
of degree two.  If $D\Phi_{[\vect{x}]}[\vect{v}]=\vect{0}$, then
$\matr{H}_{\vect{x}}\vect{v}\in
\RR\vect{q}(\vect{x})=\RR\matr{H}_{\vect{x}}\vect{x}$.
Invertibility of $\matr{H}_{\vect{x}}$ gives $[\vect{v}]=0$.
Hence the tangent bundle
$\mathrm{T}(\mathbb{RP}^5\setminus\set{Z}_{\RR})$
is isomorphic to $\Phi^*\mathrm{T}\mathbb{S}^5$ and has
second Stiefel--Whitney class zero.  If $\xi$ is the generator of
$H^1(\mathbb{RP}^5;\mathbb{Z}/2)$, then
$w(\mathrm{T}\mathbb{RP}^5)=(1+\xi)^6$ and
$w_2(\mathrm{T}\mathbb{RP}^5)=\binom62\xi^2=\xi^2$.
Its restriction to a linear projective plane is nonzero, so no such
plane avoids $\set{Z}_{\RR}$.
If $\dim_{\RR}\set{Z}_{\RR}\le2$, the incidence family of planes
meeting it would have dimension at most $2+6<9$, the dimension of
$\operatorname{Gr}(3,6)$.  Thus $\dim_{\RR}\set{Z}_{\RR}\ge3$.

Let $\set{X}=V_{\mathbb{C}}(f)\subseteq\mathbb{P}^5_{\mathbb{C}}$.
The first Euler identity shows that
$\set{Z}_{\RR}\subseteq\operatorname{Sing}\set{X}$.  Since $\set{X}$ is an
integral hypersurface in characteristic zero,
$\dim_{\mathbb{C}}\operatorname{Sing}\set{X}\le3$.
Let \(\set{Y}\) be the real Zariski closure of \(\set{Z}_{\RR}\).  Because
$\operatorname{Sing}\set{X}$ is defined over $\RR$, the complexification of
$\set{Y}$ satisfies $\set{Y}_{\mathbb C}\subseteq\operatorname{Sing}\set{X}$.
Hence \cite[Proposition~2.8.2]{bochnak1998real} and the preceding lower bound
give
$3\le\dim_{\RR}\set{Z}_{\RR}=\dim\set{Y}
\le\dim_{\mathbb C}\operatorname{Sing}\set{X}\le3$.
Thus all these dimensions are three.  Since $\set{Z}_{\RR}$ is Zariski dense
in $\set{Y}$, some three-dimensional real irreducible component $\set{Y}_0$
has $\set{Z}_{\RR}\cap\set{Y}_0$ Zariski dense in $\set{Y}_0$.
Its complexification is irreducible: otherwise it would split into distinct
conjugate components and all real points would lie in their intersection,
of dimension at most two.  Denote this complexification by
$\set{Z}\subseteq\operatorname{Sing}\set{X}$.  Then $\set{Z}$ is
conjugation-stable and $\set{Z}(\RR)$ is Zariski dense.

The hyperplane linear system on $\mathbb{P}^5_{\mathbb C}$ is base-point-free.
Applying Bertini irreducibility successively to $\set{X}$
\cite[Lemma~37.32.3]{stacks-project} and Bertini smoothness successively to
$\set{Z}_{\mathrm{reg}}$
\cite[Lemma~33.47.3]{stacks-project}, while choosing the final plane
to avoid $\operatorname{Sing}\set{Z}$, a general flag of three hyperplanes
cuts $\set{X}$ in an integral plane cubic and meets $\set{Z}$ transversely
in exactly $\deg\set{Z}$ reduced smooth points.  
Here integrality follows inductively: at each step, Bertini smoothness
applied to the regular locus gives a generically reduced hyperplane
section, since a general section is not contained in the singular locus.
As a hypersurface in the corresponding linear space, each section has
no embedded components, so it is reduced and, by geometric
irreducibility, integral.
All
these points are singular on the cubic.  An integral plane cubic has at
most one singular point, since the line through two such points would be a
component by B\'ezout.  Hence \(\deg\set{Z}=1\).  An irreducible projective
variety of degree one is linear
\cite[Chapter~I, Section~7]{hartshorne1977geometry}, so \(\set{Z}\) is a
conjugation-stable \(\mathbb P^3\), hence is defined over \(\RR\).

Choose real linear coordinates with
$I(\set{Z})=(\vect{y}_1,\vect{y}_2)$.
Since $\set{Z}\subseteq\operatorname{Sing}\set{X}$, one has
$f\in I(\set{Z})^2$: expanding \(f\) by its degree in
\((\vect{y}_1,\vect{y}_2)\), the degree-zero term and the two
degree-one coefficients vanish because \(f\) and all its first
derivatives vanish on \(\set{Z}\).  Thus
$f=\vect{y}_1^2\vect{\ell}_1
+\vect{y}_1\vect{y}_2\vect{\ell}_2
+\vect{y}_2^2\vect{\ell}_3.$ 
The affine cone $\widehat{\set{Z}}\subseteq\set{U}^*$ is a
four-dimensional vector space, so it contains a nonzero
$\vect{v}\in\bigcap_{i=1}^3\ker\vect{\ell}_i$.
Then $D_{\vect{v}}f=0$.  Since
$D_{\vect{v}}f=3(\vect{v}\mathbin{\lrcorner}\tens{T})$, the vector
$\vect{v}$ lies in the kernel of the first flattening; its injectivity for a
concise tensor forces $\vect{v}=\vect{0}$, a contradiction.
Some $\vect{x}\ne\vect{0}$ therefore satisfies
$\det\matr{H}_{\vect{x}}=0$ and $\vect{q}(\vect{x})\ne\vect{0}$.
Taking $\alpha=\vect{x}$ proves the lemma.
\end{proof}

Combining the preceding lemma with polar completion and the shifted
maximum-rank estimate gives the exceptional endpoint.

\begin{corollary}[Six-variable maximum Waring rank]
\label{cor:six-variable-maximum}
One has \(\Mthree(6)\le16\).
\end{corollary}

\begin{proof}
By \cref{prop:shifted-maximum-rank},
\(\Mthree(5)\le2\gthree(4)+1=11\).  For
\(\tens{0}\ne\tens{T}\in\Sym^3(\RR^6)\), choose \(\alpha\) from
\cref{lem:singular-transverse-six}.  Then
\(h=\rank(\alpha\mathbin{\lrcorner}\tens{T})\le5\), and the transverse
branch of \cref{lem:polar-completion} gives
$s(\tens{T})\le h+\Mthree(5)\le5+11=16.$
\end{proof}

\begin{remark}[Comparison with general maximum-rank bounds]
The Alexander--Hirschowitz theorem gives
$\gthree(6)=\lceil\binom{8}{3}/6\rceil=10$
\cite{hirschowitz1995polynomial,brambilla2008alexander}.  Consequently,
the general real maximum-rank estimate of Blekherman and Teitler yields
$\Mthree(6)\leq2\gthree(6)=20$
\cite[Corollary~9]{blekherman2015maximum}.  Writing
$M_{3,\mathbb{C}}(6)$ for the corresponding complex maximum,
Jelisiejew's estimate specializes to
$M_{3,\mathbb{C}}(6)\leq\binom{7}{2}-\binom{3}{0}=20$
\cite[Corollary~6]{jelisiejew2014upper}, while the
projective-homogeneous-variety bound of Buczyński, Han, Mella, and
Teitler, applied to the cubic Veronese variety, gives \cite[Theorem~3.9]{buczynski2018locus} 
$$M_{3,\mathbb{C}}(6)\leq2\gthree(6)-1=19.$$  
The last two estimates concern
rank over algebraically closed fields and, in particular, the complex
bound $19$ does not furnish a real-rank bound.  Thus the estimate
$\Mthree(6)\leq16$ obtained above is sharper than the cited general
real bound and is not a consequence of these complex estimates.
\end{remark}

\subsection{Assembly of the bounds}

With the hyperplane reduction, shifted maximum-rank estimate, and
six-variable endpoint established, we now assemble the two dimension
bounds.

\begin{proof}[Proof of \cref{prop:maximum-waring-gap}]
Put $d=d(\tens{T})$.  The first inequality in
\eqref{eq:uniform-polar-bound} follows from
\cref{prop:hyperplane-gap-reduction}.  For $d\geq3$ and $d\ne7$,
\cref{prop:shifted-maximum-rank} and
\eqref{eq:generic-cubic-rank} give
$$\Mthree(d-1)-d+1
\leq2\lceil d(d-1)/6\rceil-d+2
\leq(d^2-4d+10)/3.$$
The last inequality uses the fact that $d(d-1)$ is congruent to $0$
or $2$ modulo $6$.  The cases $d=1$ and $d=2$ follow directly from
$\Mthree(0)=0$ and $\Mthree(1)=1$, respectively.  For $d=7$,
\cref{cor:six-variable-maximum} gives
$\Mthree(6)-6\le10\le31/3$.  This proves the second inequality in
\eqref{eq:uniform-polar-bound}.
\end{proof}

\begin{remark}[The seven-dimensional transition]
\label{rem:dimension-seven-transition}
The two dimension estimates have different inputs at $d=7$.  Indeed,
$\Cdsharp(7)=\min\{\Vthree(6)-6,
\lfloor3\Vthree(7)/4\rfloor\}=11$, whereas
$(7^2-4\cdot7+10)/3=31/3$.  Thus the inequality in
\cref{prop:maximum-waring-gap} at $d=7$ does not follow by replacing
$\Mthree(6)$ with $\Vthree(6)=17$.  It uses the separate estimate
$\Mthree(6)\le16$ from \cref{cor:six-variable-maximum}, which gives
$\DeltaT(\tens{T})\le16-7+1=10\le31/3$.
\end{remark}

\begin{proof}[Proof of \cref{prop:dimension-quadratic}]
If $\tens{T}=\tens{0}$, then $d=0$ and the claim is immediate.
Henceforth $d\geq1$.  \Cref{prop:maximum-waring-gap} and
\cref{prop:shifted-maximum-rank} give
$\DeltaT(\tens{T})\leq\Vthree(d-1)-d+1$, with the conventions
$\Mthree(0)=\Vthree(0)=0$ when $d=1$.
On the other hand, symmetrizing a minimum CP decomposition and applying
\eqref{eq:four-cube} gives $s(\tens{T})\leq4r(\tens{T})$.
By \cref{lem:concise-reduction,prop:shifted-maximum-rank},
$s(\tens{T})\leq\Mthree(d)\leq\Vthree(d)$.  Hence
$$\DeltaT(\tens{T})=s(\tens{T})-r(\tens{T})
\leq\frac{3s(\tens{T})}{4}\leq\frac{3\Vthree(d)}{4}.$$  
Since the gap is an integer,
taking the better
of these two estimates proves $\DeltaT(\tens{T})\le \Cdsharp(d)$.

Expanding the ceiling in \eqref{eq:generic-cubic-rank} gives the exact
branch comparison
\begin{equation*}
\Cdsharp(d)=
\begin{cases}
\Vthree(d-1)-d+1,&1\leq d\leq18,\\
\left\lfloor\dfrac{3\Vthree(d)}4\right\rfloor,&d\geq19.
\end{cases}
\end{equation*}
The two branches agree for $d=17,18$.  For $d\geq19$ one has
$\Cdsharp(d)=(d^2+d+\kappa_{d\bmod12})/4$, where
\[
\begin{array}{c|rrrrrrrrrrrr}
d\bmod12&0&1&2&3&4&5&6&7&8&9&10&11\\ \hline
\kappa&0&6&2&0&4&2&2&4&0&2&6&0.
\end{array}
\]
A direct check for $1\leq d\leq18$, and the inequality
$\kappa_{d\bmod12}\leq6$ for $d\geq19$, give
$\Cdsharp(d)\leq d^2/4+d/4+3/2$.
Since $d\leq n$, the last inequality in
\cref{prop:dimension-quadratic} follows.
\end{proof}

\begin{proof}[Proof of \cref{thm:essential-dimension-main}]
If $d(\tens{T})=0$, then $\tens{T}=\tens{0}$ and the assertion is immediate.
For $1\leq d(\tens{T})\leq18$, use
\cref{prop:maximum-waring-gap}; for $d(\tens{T})\geq19$, use
\cref{prop:dimension-quadratic}.  These cases give exactly the definition
\eqref{eq:def-C-dimension}.
\end{proof}
The preceding computation also gives
$\Cdsharp(d)=d^2/4+d/4+O(1)$ as $d\to\infty$.

\section{Conclusions}
\label{sec:discussion}

Motivated by the rank-comparison problem underlying Comon's conjecture,
we have studied the quantitative gap
$\DeltaT(\tens{T})=s(\tens{T})-r(\tens{T})$ for real symmetric
third-order tensors.  Our results provide two complementary types of
control.  
% In terms of the rank excess,
% \cref{thm:rank-excess-main} gives
% $\DeltaT(\tens{T})\leq \Ce(e(\tens{T}))$; explicitly,
% \[
% \DeltaT(\tens{T})\leq
% \begin{cases}
% 0,&0\leq e(\tens{T})\leq3,\\
% 9e(\tens{T})-3,&e(\tens{T})\geq4.
% \end{cases}
% \]
% The first branch follows from low-wedge partialization and saturated
% descent; the general wedge-support count gives the second. 
% In terms of the essential dimension,
% \cref{thm:essential-dimension-main} gives
% $\DeltaT(\tens{T})\leq \Cd(d(\tens{T}))$.
% The first branch is furnished by
% \cref{prop:maximum-waring-gap}:
% $\DeltaT(\tens{T})\leq
% \Mthree(d(\tens{T})-1)-d(\tens{T})+1$ for nonzero tensors.  This
% estimate comes from the polar-completion and adaptive-hyperplane
% lemmas. 
% The second branch follows from \cref{prop:dimension-quadratic} and gives
% $\DeltaT(\tens{T})\leq d(\tens{T})^2/4+d(\tens{T})/4+3/2$.
% The Alexander--Hirschowitz theorem determines the auxiliary generic
% complex Waring rank
% \cite{hirschowitz1995polynomial,brambilla2008alexander}, and
% \cref{prop:shifted-maximum-rank} converts it into the real estimate
% $\Mthree(d)\leq2\gthree(d-1)+1$ used here. 
In fact, these two bounds are upper estimates rather than assertions of optimal
growth.  Improving maximum real Waring-rank estimates would sharpen
the dimension-controlled estimate, while improving the first likely requires finer
control of the wedge-support and direction configurations. 
The results obtained in this paper do not resolve the following problem, 
which may be viewed as a weaker form of Comon's conjecture.

\begin{problem}
Does there exist a positive constant $c>0$ such that
$\DeltaT(\tens{T})\leq c$
for every real symmetric tensor $\tens{T}$ of order three, independently of $d(\tens{T})$ and $r(\tens{T})$?
\end{problem}

\noindent\textbf{Acknowledgments.} 
The author acknowledges the assistance of ChatGPT 5.6 Sol throughout the mathematical development and preparation of this manuscript. 
The author takes full responsibility for the content of this paper. 
This work was supported by the Guangdong Basic and Applied Basic
Research Foundation (2026A1515050001, 2026A1515010188).

\bibliographystyle{siamplain}
\bibliography{ReferenceTensor}

\end{document}